\documentclass[a4paper,11pt,reqno]{amsart}
\usepackage{latexsym,amsmath,amssymb,amsfonts}
\usepackage{a4wide}
\usepackage{graphicx}
\usepackage[usenames, dvipsnames]{xcolor}
\usepackage[colorlinks=true, linkcolor=blue, citecolor=ForestGreen]{hyperref}

\newcommand{\xupref}[2]{\hspace{-0.3ex}\stackrel{\eqref{#1}}{#2}} 
\newcommand{\upupref}[3]{\hspace{-3ex}\stackrel{\eqref{#1},\eqref{#2}}{#3}}

\theoremstyle{plain}
\begingroup
\newtheorem{theorem}{Theorem}[section]
\newtheorem{lemma}[theorem]{Lemma}

\endgroup
\theoremstyle{definition}
\begingroup

\endgroup
\theoremstyle{remark}
\newtheorem{remark}[theorem]{Remark}

\numberwithin{equation}{section}

\newcommand{\e}{\varepsilon}

\newcommand{\Z}{\mathbb Z}
\newcommand{\R}{\mathbb R}

\newcommand{\de}{\,\mathrm{d}}

\renewcommand{\setminus}{\backslash}
\newcommand{\inte}{\int_{-\infty}^\infty}

\newcommand{\expyz}{e^{\frac{y-z}{\rho}}}
\newcommand{\expay}{e^{-\frac{a}{\rho}y}}
\newcommand{\expaz}{e^{-\frac{a}{\rho}z}}
\newcommand{\cbar}{\bar{c}}

\title[Self-similar solutions to coagulation equations with time-dependent tails]{Self-similar solutions to coagulation equations\\ with time-dependent tails: the case of homogeneity smaller than one}
\author{Marco Bonacini \and Barbara Niethammer \and Juan J. L. Vel\'{a}zquez}
\date{\today}
\address{Rheinische Friedrich-Wilhelms-Universit\"at Bonn, Institut f\"ur Angewandte Mathematik, Endenicher Allee 60, 53115 Bonn, Germany}
\email{bonacini@iam.uni-bonn.de; niethammer@iam.uni-bonn.de; velazquez@iam.uni-bonn.de}


\begin{document}

\begin{abstract}
We prove the existence of a one-parameter family of self-similar solutions with time-dependent tails for Smoluchowski's coagulation equation, for a class of rate kernels $K(x,y)$ which are homogeneous of degree $\gamma\in(-\infty,1)$ and satisfy $K(x,1)\sim x^{-a}$ as $x\to 0$, for $a=1-\gamma$. In particular, for small values of a parameter $\rho>0$ we establish the existence of a positive self-similar solution with finite mass and asymptotics $A(t)x^{-(2+\rho)}$ as $x\to\infty$, with $A(t)\sim\rho t^\frac{\rho}{1-\gamma}$.
\end{abstract}

\maketitle


\section{Introduction} \label{sect:intro}

Smoluchowski's coagulation equation is a widely used mean-field model, originally derived by Smoluchowski in 1916 \cite{Smo27}, for the description of general coalescence phenomena where particles grow by successive mergers. It applies to homogeneous dilute systems of clusters which are fully identified by their mass and can coagulate through binary collisions. If $f(x,t)$ denotes the number density of clusters of size $x>0$ at time $t$, then the dynamics of $f$ is governed by the equation
\begin{equation} \label{eq:smol}
\partial_t f(x,t) = \frac12\int_0^x K(x-y,y)f(x-y,t)f(y,t)\de y - \int_0^\infty K(x,y)f(x,t)f(y,t)\de y\,,
\end{equation}
where the microscopic details of the specific merging process are contained in the so-called rate kernel $K$.
We refer to the surveys \cite{LM04,Ley03} for a basic physical introduction and an overview of mathematical results on coagulation models, as well as to the references therein.

In this paper we are interested in the study of the class of homogeneous kernels of degree $\gamma\in(-\infty,1)$ characterized by
\begin{equation} \label{kernel0}
K(x,1) \sim \frac{1}{x^a} \quad\text{ as }x\to0,\qquad a=1-\gamma\,.
\end{equation}
This includes, for instance, the general sum kernel $K(x,y)=yx^{-a}+xy^{-a}$.
Such kernels have already been reported as peculiar in the physics literature, see in particular Section~4.2 in \cite{vDE88}, and \cite{vD87}.

A central question in the understanding of the dynamics of \eqref{eq:smol} is whether solutions with finite mass exhibit a universal self-similar form as time tends to infinity:
\begin{equation} \label{selfsim}
f(x,t) \sim f_S(x,t) = \frac{1}{s(t)^2}\Phi\Bigl(\frac{x}{s(t)}\Bigr)\,,
\qquad t\to\infty,
\end{equation}
with the mean particle size $s(t)\to\infty$ as $t\to\infty$ and the self-similar profile $\Phi$ to be determined, depending on the coagulation kernel $K$ but not on the specific initial datum. Despite several formal computations supporting this hypothesis, only partial results are available from the rigorous point of view. These include in particular the case of the three solvable kernels (the constant kernel $K(x,y)=2$, the additive kernel $K(x,y)=x+y$ and the multiplicative kernel $K(x,y)=xy$), for which self-similar solutions can be computed explicitly via Laplace transform and their domain of attraction can be fully characterized, see \cite{CMM10,MP04,Sri11}. However, the existence of self-similar solutions with finite mass has been established also for a large class of kernels homogeneous of degree $\gamma<1$, see in particular \cite{EMR05,FL05}, and some properties of these solutions have been investigated in \cite{CanMisch11,EM06,FL06}. More recently, the first existence results of \textit{fat tail} solutions were obtained in \cite{NV11} for the diagonal kernel, in \cite{NV13} for homogeneous kernels of degree $\gamma\in[0,1)$ satisfying the bound $K(x,y)\leq C(x^\gamma+y^\gamma)$, and in \cite{NTV16} for a class of singular kernels homogeneous of degree $\gamma\in(-\infty,1)$ such that
\begin{equation} \label{eq:intro1}
C_1\bigl( x^{-a}y^b + x^by^{-a} \bigr) \leq K(x,y) \leq C_2 \bigl( x^{-a}y^b + x^by^{-a} \bigr)\,,
\qquad a>0,\, b<1,\, \gamma=b-a,
\end{equation}
including for instance the classical Smoluchowski's kernel. Much less is known about the question of uniqueness of self-similar profiles with given decay behaviour at infinity, which has so far been answered only for a the special case of kernels close to the constant one \cite{NTV16b}.

It is the purpose of this work to establish the existence of a new one-parameter family of nonnegative self-similar solutions with finite mass and time-dependent tails (see \eqref{tails}) for the class of kernels homogeneous of degree $\gamma\in(-\infty,1)$ satisfying \eqref{kernel0}. In particular, this corresponds to the choice $b=1$ in \eqref{eq:intro1}, a case that has not been treated before.
In our main result, see Theorem~\ref{thm:exist}, we show indeed that for every small value of a parameter $\rho\in(0,\rho_*)$, for some $\rho_*>0$, there is such a solution whose asymptotic behaviour at infinity is characterized in terms of $\rho$.

We remark that we performed a similar analysis in \cite{BNV} for rate kernels homogeneous of degree one and not diagonally dominant, obtaining by the same methods used here an analogous family of self-similar solutions with time-dependent tails. However, a significant difference is that in the case considered here the solution exhibits a very strong variation in a small transition layer. This can be seen by working in exponential variables, see \eqref{variables}: in these variables the solution grows rapidly from values close to 0 to values of order one in a small interval of order $\rho$ around the origin. Such behaviour makes the analysis more involved than in the case treated in \cite{BNV}.

\medskip
\subsection*{Self-similar solutions}
We now formulate the precise assumptions on the rate kernels that we consider in this work: $K$ is a continuous, nonnegative and symmetric map, homogeneous of degree $\gamma\in(-\infty,1)$:
\begin{equation} \label{kernel1}
K\in C((0,\infty)\times(0,\infty)), \qquad
K(x,y) = K(y,x) \geq0 \quad\text{for all }x,y\in(0,\infty),
\end{equation}
\begin{equation} \label{kernel2}
K(\alpha x,\alpha y) = \alpha^\gamma K(x,y) \qquad\text{for all }x,y\in(0,\infty),\, \alpha>0.
\end{equation}
Moreover, we assume that the following stronger version of \eqref{kernel0} holds:
\begin{equation} \label{kernel3}
|x^a K(x,1) - 1| \leq K_0 x^\delta
\qquad\text{for every }x\in(0,2),
\qquad a=1-\gamma
\end{equation}
for constants $K_0>0$ and $\delta\in(0,1)$ (we also assume without loss of generality that $\delta<a$).
Notice that, by homogeneity, \eqref{kernel3} also settles the behaviour of the kernel $K(\cdot,1)$ at infinity: indeed,
$K(x,1) = x^\gamma K(1,\frac{1}{x}) \sim x^\gamma x^a = x$ as $x\to\infty$, and more precisely
\begin{equation}\label{kernel4}
K(x,1)\leq (K_0+1) x
\qquad\text{for every }x>1.
\end{equation}

Self-similar solutions with finite mass have the form \eqref{selfsim} for suitable functions $s(t)$ (representing the growth of the average particle size) and $\Phi$ (the self-similar profile): by plugging this \textit{ansatz} into \eqref{eq:smol} and using the fact that the kernel has homogeneity $\gamma=1-a$, we obtain an explicitly solvable ordinary differential equation for $s(t)$
\begin{equation*}
\frac{\de s}{\de t} = b s(t)^\gamma\,,
\end{equation*}
and that the self-similar profile $\Phi$ solves the integro-differential equation
\begin{equation} \label{eq:selfsim0}
b\bigl[ -2\Phi(\xi) - \xi\Phi'(\xi) \bigr] = \frac12\int_0^\xi K(\xi-\eta,\eta)\Phi(\xi-\eta)\Phi(\eta)\de \eta - \int_0^\infty K(\xi,\eta)\Phi(\xi)\Phi(\eta)\de\eta\,,
\end{equation}
for a real constant $b>0$. We get for the mean particle size
\begin{equation} \label{selfsimbis}
s(t) = \bigl((1-\gamma)bt\bigr)^\frac{1}{1-\gamma}\,.
\end{equation}
We next define the notion of a weak solution $\Phi$ to \eqref{eq:selfsim0}: multiplying the equation by $\xi$, after a change of variables and a (formal) application of Fubini's Theorem we obtain that $\Phi$ solves 
\begin{equation*}
b \partial_\xi\bigl( \xi^2 \Phi(\xi) \bigr) = \partial_\xi\biggl( \int_0^\xi\int_{\xi-\eta}^\infty K(\eta,\zeta)\eta\Phi(\eta)\Phi(\zeta)\de \zeta \de \eta \biggr)\,,
\end{equation*}
which, after integration in $\xi$, yields the weaker form of the equation for a self-similar profile:
\begin{equation} \label{eq:selfsim}
b \xi^2 \Phi(\xi) = \int_0^\xi\int_{\xi-\eta}^\infty K(\eta,\zeta)\eta\Phi(\eta)\Phi(\zeta)\de \zeta \de \eta\,.
\end{equation}
We also introduce the mass and the $\gamma$-moment of $\Phi$:
\begin{equation*}
M(\Phi) := \int_0^\infty \xi\Phi(\xi) \de \xi\,,
\qquad
M_\gamma(\Phi) := \int_0^\infty \xi^{\gamma}\Phi(\xi)\de \xi\,.
\end{equation*}

\begin{remark}[Scaling] \label{rm:scaling}
Equation \eqref{eq:selfsim} has the following scale-invariance property: if $\Phi$ solves \eqref{eq:selfsim}, then the rescaled function $\widetilde{\Phi}(\xi) = \lambda \Phi(\mu \xi)$, for $\lambda,\mu >0$, solves
\begin{equation*}
\frac{\lambda b}{\mu^{1+\gamma}} \xi^2 \widetilde{\Phi}(\xi) = \int_0^\xi\int_{\xi-\eta}^\infty K(\eta,\zeta)\eta\widetilde{\Phi}(\eta)\widetilde{\Phi}(\zeta)\de \zeta \de \eta
\end{equation*}
with
\begin{equation*}
M(\widetilde{\Phi}) = \frac{\lambda}{\mu^2} M(\Phi)\,,
\qquad
M_\gamma(\widetilde{\Phi}) = \frac{\lambda}{\mu^{1+\gamma}}M_\gamma(\Phi)\,.
\end{equation*}
Notice in particular that the coefficient $b$ in the equation and the $\gamma$-moment of the solution scale with the same factor - in other words, the equation is left invariant by a rescaling that fixes the $\gamma$-moment. By choosing the two parameters $\lambda$ and $\mu$ we can then normalize the mass and the $\gamma$-moment of the solution.
\end{remark}

\medskip
\subsection*{Main result}
By formal asymptotics (see Section~\ref{sect:asympt}), one finds that there is a one-to-one correspondence between the coefficient $b$ in equation \eqref{eq:selfsim} and the decay behaviour of the corresponding solution $\Phi$. In particular, if $\Phi(\xi)\sim \xi^{-(2+\rho)}$ as $\xi\to\infty$, for $\rho>0$, then
\begin{equation} \label{b}
b = \frac{1+\rho}{\rho} M_\gamma(\Phi)\,.
\end{equation}
Therefore, in view of Remark~\ref{rm:scaling} without loss of generality we look for a normalized solution $\Phi$ to the following problem:
\begin{equation} \label{eq:selfsim2}
\xi^2 \Phi(\xi) = \int_0^\xi \int_{\xi-\eta}^\infty K(\eta,\zeta)\eta \Phi(\eta)\Phi(\zeta)\de \zeta \de \eta\,,
\qquad\text{with } M(\Phi)=1,\quad M_\gamma(\Phi) = \frac{\rho}{1+\rho}\,.
\end{equation}
The main finding of the paper is the following.

\begin{theorem} \label{thm:exist}
Assume that the kernel $K$ satisfies assumptions \eqref{kernel1}, \eqref{kernel2} and \eqref{kernel3}.
Then there exists $\rho_*>0$ such that for every $\rho\in(0,\rho_*)$
there is a positive, continuous solution $\Phi_\rho$ to \eqref{eq:selfsim2}, satisfying in addition
\begin{align}
	\frac{c_1\rho}{\xi^{2}}e^{-\frac{2}{a}\xi^{-a}} &\leq \Phi_\rho(\xi) \leq \frac{c_2 \rho}{\xi^{2}}e^{-\frac{1}{2a}\xi^{-a}} & &\text{as }\xi\to0, \label{asymp1}\\
	\Phi_\rho(\xi) & = \frac{k_\rho\rho}{\xi^{2+\rho}} + o\Bigr(\frac{1}{\xi^{2+\rho}}\Bigr)& &\text{as }\xi\to\infty, \label{asymp2}
\end{align}
for a constant $k_\rho\to1$ as $\rho\to0$.
\end{theorem}

It is worth to remark that the self-similar solutions obtained in Theorem~\ref{thm:exist} have \textit{time-dependent tails}. Indeed, the asymptotics as $x\to\infty$ of the self-similar solution $f_S$ corresponding to the self-similar profile $\Phi_\rho$, according to \eqref{selfsim} and \eqref{selfsimbis}, is
\begin{equation} \label{tails}
f_S(x,t)\sim \frac{A(t)}{x^{2+\rho}}\,, \qquad A(t)= k_\rho\rho ((1-\gamma)t)^\frac{\rho}{1-\gamma}\,.
\end{equation}
In this respect, they differ from the fat tail solutions whose existence has been established so far for kernels with homogeneity $\gamma<1$ in \cite{NTV16,NV13}.

By analogy with the case of homogeneity equal to one, we expect that there is a critical $\rho_{\rm crit}\in (0,\infty]$ such that solutions as in Theorem~\ref{thm:exist} exist for any $\rho \in (0,\rho_{\rm crit})$, and in addition there is a solution bounded exponentially for $\rho =\rho_{\rm crit}$. A conjecture of the precise decay of such a solution is given in \cite{vD87} via self-consistent arguments.

The proof of Theorem~\ref{thm:exist} is carried out in a sequence of steps and results from the combination of Theorem~\ref{thm:fixedpoint}, Theorem~\ref{thm:decay-}, and Theorem~\ref{thm:decay+} below (whose statements are formulated in different variables, see \eqref{variables}). We will explain at the end of Section~\ref{sect:decay+} how to bring together all the intermediate results in order to get the final statement. The strategy leading to Theorem~\ref{thm:exist} is close in spirit to that developed in \cite{BNV} for a similar analysis in the case of homogeneous kernels of degree one. 

The first step (Theorem~\ref{thm:fixedpoint}) consists in establishing the existence of a continuous solution to \eqref{eq:selfsim2}
and mainly rests on an application of Banach Fixed Point Theorem in a space of functions with suitable decay at $\pm\infty$ (in the new variables).
This is carried out in Section~\ref{sect:fixedpoint}, which contains the main argument, and Section~\ref{sect:estimates}, where many technical estimates needed along the proof are collected.

As a second step, we determine in Section~\ref{sect:decay-} the exact decay behaviour of the solution at the origin (Theorem~\ref{thm:decay-}). Notice that \eqref{eq:decay-} is actually a more precise condition than \eqref{asymp1}, but has the disadvantage of being expressed in terms of an implicit function $\psi_\rho$ depending on the solution itself (see in particular Remark~\ref{rm:decay-} for more details).

Finally, in Section~\ref{sect:decay+} we find the exact asymptotics of the solution at infinity (Theorem~\ref{thm:decay+}). This is significantly more involved than the proof of the asymptotics at the origin, and requires in particular an additional estimate on the Lipschitz constant of the solution.


\section{Heuristics} \label{sect:asympt}

We present a formal justification for the relation \eqref{b} between the coefficient $b$ in the equation \eqref{eq:selfsim} and the exponent $\rho$ describing the decay behaviour of the solution at infinity.
Let us make the ansatz that a self-similar profile $\Phi$ exists with $\Phi(\xi)\sim k(\rho)\xi^{-(2+\rho)}$ as $\xi\to\infty$.
We formally compute the asymptotics  as $\xi\to\infty$ of the integral
\begin{equation*}
\int_0^\xi\int_{\xi-\eta}^\infty K(\eta,\zeta)\eta\Phi(\eta)\Phi(\zeta)\de \zeta \de \eta
\end{equation*}
appearing on the right-hand side of \eqref{eq:selfsim}. In order to exploit the assumption \eqref{kernel3} on the kernel, we examine separately the contributions of the regions where $\zeta\leq\e\eta$, $\zeta\geq\frac{1}{\e}\eta$ for a small $\e>0$ (the contribution of the remaining part being negligible, as we will see).

In the case $\zeta\leq\e\eta$, the variable $\eta$ is close to $\xi$ (up to errors converging to zero as $\e\to0$), and by homogeneity of the kernel and by \eqref{kernel0} we have $K(\eta,\zeta)\sim\eta\zeta^{-a}$. Then
\begin{equation*}
\begin{split}
\int_{\frac{\xi}{1+\e}}^\xi \de \eta \int_{\xi-\eta}^{\e\eta} & K(\eta,\zeta) \eta\Phi(\eta)\Phi(\zeta)\de\zeta
\sim k(\rho)\xi^2\xi^{-(2+\rho)} \int_{\frac{\xi}{1+\e}}^\xi \de \eta \int_{\xi-\eta}^{\e\eta} \zeta^{-a}\Phi(\zeta)\de\zeta \\
& = k(\rho)\xi^{-\rho} \int_{0}^{\frac{\e\xi}{1+\e}} \de\zeta \, \zeta^{-a}\Phi(\zeta) \int_{\xi-\zeta}^\xi \de\eta
+ k(\rho)\xi^{-\rho} \int_{\frac{\e\xi}{1+\e}}^{\e \xi} \de\zeta \, \zeta^{-a}\Phi(\zeta) \int_{\frac{\zeta}{\e}}^\xi \de \eta \\
& = k(\rho)\xi^{-\rho} \int_{0}^{\frac{\e\xi}{1+\e}} \zeta^{\gamma}\Phi(\zeta)\de \zeta
+ k(\rho)\xi^{-\rho} \int_{\frac{\e\xi}{1+\e}}^{\e \xi}\zeta^{-a}\Phi(\zeta)\Bigl(\xi-\frac{\zeta}{\e}\Bigr)\de\zeta\,.
\end{split}
\end{equation*}
In the limit as $\xi\to\infty$ the first term converges to $k(\rho)M_\gamma(\Phi)\xi^{-\rho}$, while the second term gives a higher order contribution, since using the power law for $\Phi$ we have
\begin{equation*}
\bigg| \xi^{-\rho} \int_{\frac{\e\xi}{1+\e}}^{\e \xi}\zeta^{-a}\Phi(\zeta)\Bigl(\xi-\frac{\zeta}{\e}\Bigr)\de\zeta\bigg|
\leq C \xi^{-\rho}\xi^{-a}\xi^{-(2+\rho)}\xi^2 \leq C\xi^{-2\rho-a}\,.
\end{equation*}

We next consider the region where $\zeta\geq\frac{1}{\e}\eta$. Here we can use the approximation of $\Phi(\zeta)$ by means of the power law, since $\zeta$ is large; moreover by the properties of the kernel we have $K(\eta,\zeta)\sim\zeta\eta^{-a}$. Then
\begin{equation*}
\begin{split}
\int_{0}^\xi \de \eta & \int_{\max\{\xi-\eta,\frac{\eta}{\e}\}}^{\infty} K(\eta,\zeta) \eta\Phi(\eta)\Phi(\zeta)\de\zeta
\sim k(\rho) \int_{0}^\xi \de \eta \int_{\max\{\xi-\eta,\frac{\eta}{\e}\}}^{\infty} \eta^\gamma\zeta\Phi(\eta) \zeta^{-(2+\rho)}\de\zeta \\
& = k(\rho) \int_{0}^{\frac{\e\xi}{1+\e}} \de \eta \, \eta^\gamma\Phi(\eta) \int_{\xi-\eta}^{\infty} \zeta^{-(1+\rho)}\de\zeta
+ k(\rho) \int_{\frac{\e\xi}{1+\e}}^\xi \de \eta \, \eta^\gamma\Phi(\eta) \int_{\frac{\eta}{\e}}^{\infty} \zeta^{-(1+\rho)}\de\zeta\,.
\end{split}
\end{equation*}
The contribution of the second term is negligible as before. For the first term we have instead
\begin{equation*}
\begin{split}
k(\rho) \int_{0}^{\frac{\e\xi}{1+\e}} \de \eta \, \eta^\gamma\Phi(\eta) \int_{\xi-\eta}^{\infty} \zeta^{-(1+\rho)}\de\zeta
& = \frac{k(\rho)}{\rho} \int_{0}^{\frac{\e\xi}{1+\e}} \eta^\gamma\Phi(\eta)(\xi-\eta)^{-\rho}\de\eta \\
& \sim \frac{k(\rho)}{\rho} \xi^{-\rho} \int_{0}^{\frac{\e\xi}{1+\e}} \eta^\gamma\Phi(\eta)\de\eta
\sim \frac{k(\rho)}{\rho}M_\gamma(\Phi)\xi^{-\rho}\,.
\end{split}
\end{equation*}

Finally it remains to consider the contribution from the region where $\e\eta<\zeta<\frac{1}{\e}\eta$: in this case we can use the power law for $\Phi$ to deduce that this term is also of order $\xi^{-2\rho-a}$, therefore negligible.
By collecting all the contributions, we obtain that the approximation of \eqref{eq:selfsim} for $\xi\to\infty$ is
\begin{equation*}
b k(\rho) \xi^{-\rho} \sim k(\rho)M_\gamma(\Phi)\xi^{-\rho} + \frac{k(\rho)}{\rho}M_\gamma(\Phi)\xi^{-\rho}\,,
\end{equation*}
which provides a formal justification for \eqref{b}.


\section{Existence of self-similar profiles via fixed point} \label{sect:fixedpoint}

In this section we formulate the issue of the existence of a self-similar profile solving equation \eqref{eq:selfsim2} as a fixed point problem, and we show in Theorem~\ref{thm:fixedpoint} how it can be solved by Banach's Contraction Theorem. We postpone to the following section all the technical estimates needed along the proof.

It is convenient to introduce a new coordinate system: we set
\begin{equation} \label{variables}
\rho h(x)= \xi^2 \Phi(\xi)\,,\qquad
\xi = e^{\frac{x}{\rho}}\,.
\end{equation}
In the new variables the equation \eqref{eq:selfsim2} for a self-similar profile $\Phi$ becomes
\begin{equation} \label{equation}
h(x) = \frac{1}{\rho} \int_{-\infty}^x \de y \int_{x+\rho\ln(1-e^{\frac{y-x}{\rho}})}^\infty  e^{-\frac{a}{\rho}z} K(e^{\frac{y-z}{\rho}},1) h(y)h(z) \de z\,,
\end{equation}
with the two constraints
\begin{equation} \label{constraints}
M(h) := \int_{-\infty}^\infty h(x)\de x =1\,,
\qquad
M_\gamma(h) := \frac{1}{\rho} \int_{-\infty}^\infty e^{-\frac{a}{\rho}x}h(x)\de x = \frac{1}{1+\rho}\,.
\end{equation}
The domain of integration in the right-hand side of \eqref{equation} is depicted in Figure~\ref{fig:domains}, left.
It will be sometimes convenient to split the region $\Omega_\rho$ below the $y$-axis into two subregions, due to the different behaviour of the kernel (Figure~\ref{fig:domains}, right):
\begin{align} \label{eq:regions}
\Omega_\rho &:= \bigl\{ (y,z)\in\R^2 \,:\, y< x, \, x+\rho\ln(1-e^{\frac{y-x}{\rho}}) < z < x \bigr\}\,, \nonumber\\
A_\rho &:= \bigl\{ (y,z)\in\Omega_\rho \,:\, x-\rho\ln2< z < x\bigr\}\,,\\
B_\rho &:= \bigl\{ (y,z)\in\Omega_\rho \,:\, z < x-\rho\ln2 \bigr\}\,. \nonumber
\end{align}
For future reference we gather here all the needed properties of the kernel in the regions of integration defined above: firstly, if $z<x-\rho\ln2$, then
\begin{equation} \label{regions2}
e^{\frac{z-x}{\rho}}\in(0,\textstyle\frac12)\,,
\qquad\qquad
x-\rho\ln2 < x+\rho\ln(1-e^{\frac{z-x}{\rho}})<x\,;
\end{equation}
this is in particular the case if $(y,z)\in B_\rho$. In $A_\rho$ and $B_\rho$ the kernel is bounded by assumption \eqref{kernel3} and by \eqref{kernel4} respectively as follows:
\begin{align}
e^{-\frac{a}{\rho}z} K(e^{\frac{y-z}{\rho}},1) \leq (1+2^\delta K_0)e^{-\frac{a}{\rho}y} \qquad\text{for every }(y,z)\in A_\rho, \label{kernelA}\\
K(e^{\frac{y-z}{\rho}},1)\leq (1+K_0)e^{\frac{y-z}{\rho}} \qquad\text{for every }(y,z)\in B_\rho, \label{kernelB}
\end{align}
and moreover, using one more time \eqref{kernel3} and the homogeneity of the kernel, one also has
\begin{equation} \label{regions3}
|K(e^{\frac{y-z}{\rho}},1) -e^{\frac{y-z}{\rho}}| \leq K_0 e^{\frac{1-\delta}{\rho}(y-z)} \qquad\text{for }(y,z)\in B_\rho\,.
\end{equation}

\begin{figure}
	\centering
	\includegraphics{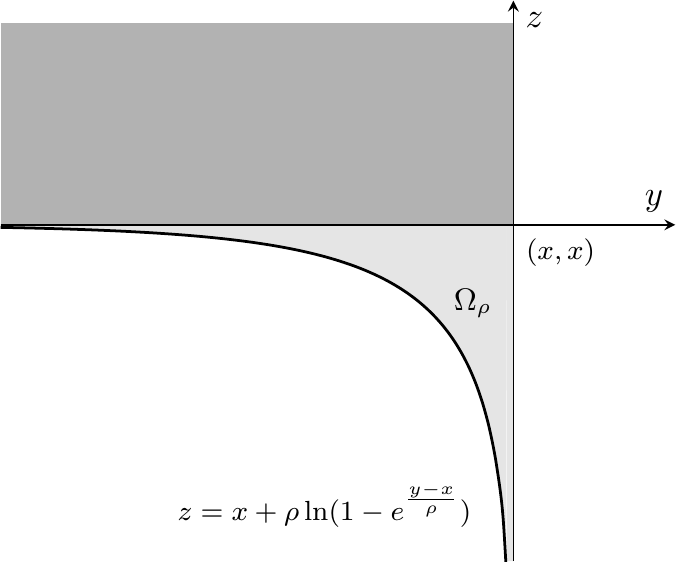}
	\hspace{4ex}
	\includegraphics{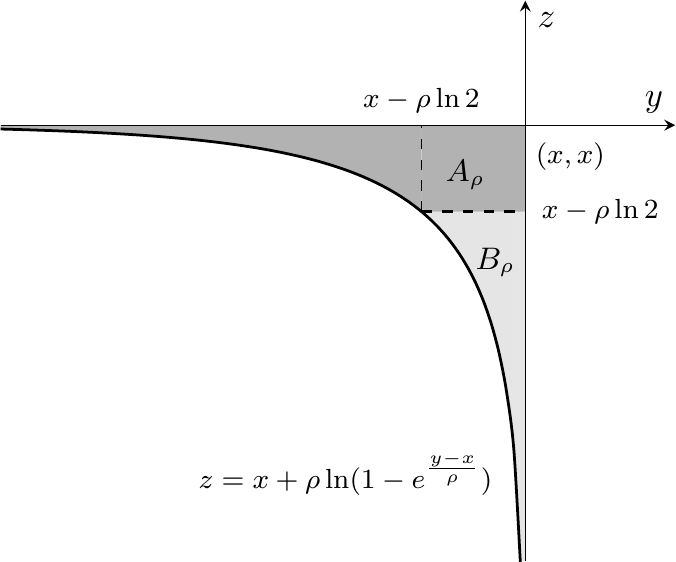}
	\caption{Left: the domain of integration in the right-hand side of equation \eqref{equation}, made up of the two regions above and below the $y$-axis (the second one denoted by $\Omega_\rho$). Right: the two regions $A_\rho$, $B_\rho$ in which we divide the domain $\Omega_\rho$, where the kernel has different decay behaviours; see \eqref{eq:regions}.}
	\label{fig:domains}
\end{figure}

\subsection*{Definition of the parameters}
We now define the various parameters appearing in the proof, pointing out the mutual dependences between them. First recall that $K_0$, $a=1-\gamma$ and $\delta$ are the coefficients appearing in \eqref{kernel3}, which depend only on the properties of the kernel. We fix three constants
\begin{equation} \label{parameters}
m>\max\{a+1,3a \}, \quad L_0>\frac{1}{a}\ln (4m), \quad \rho_0\in(0,1). 
\end{equation}
In addition, we introduce two more free parameters
\begin{equation}\label{parameters2}
L\geq L_0, \quad \rho\in(0,\rho_0),
\end{equation}
which will be chosen at the end of the proof, depending on all the other constants. In particular, $L$ will be chosen sufficiently large, and in turn $\rho$ small, depending also on $L$.
Throughout the paper it will always be assumed that $\rho$ and $L$ satisfy the additional condition
\begin{equation}\label{parameters2bis}
\rho e^{(2a+2)L} < c_0\,,
\end{equation}
where $c_0$ is the constant defined in Lemma~\ref{lem:Q} (depending on all the other parameters).
Finally, we will denote by $C$ a generic constant, possibly changing from line to line, depending only on $K_0$, $a$, $\delta$, $m$, $L_0$ and $\rho_0$, but not on $\rho\in(0,\rho_0)$ and $L>L_0$.

For $L$ and $\rho$ satisfying \eqref{parameters2} and \eqref{parameters2bis}, we introduce the space
\begin{equation}\label{space}
X_{L,\rho} := \Bigl\{ h\in C(\R\setminus\{L\rho\}) \,:\, M(h)=1,\, M_\gamma(h)=\textstyle\frac{1}{1+\rho},\, |h(x)|\leq 2 e^{\omega(x)} \Bigr\}
\end{equation}
with norm
\begin{equation*}\label{norm}
\|h\| := \sup_{x\in\R}\frac{|h(x)|}{e^{\omega(x)}}\,,
\end{equation*}
where $M(h)$ and $M_\gamma(h)$ are the quantities appearing in \eqref{constraints} and the exponent $\omega(x)$ is defined as
\begin{equation}\label{parameters3}
\omega(x):=
\begin{cases}
\frac{m}{\rho}(x+L\rho) & \text{if }x\leq-L\rho,\\
0 & \text{if }-L\rho < x < L\rho,\\
-\frac12(x-L\rho) & \text{if }x\geq L\rho.
\end{cases}
\end{equation}
We will look for a solution to \eqref{equation} in the class $X_{L,\rho}$ as a fixed point of a suitable operator, which we now define.

\subsection*{Approximation of the equation}
We first introduce, for $h\in X_{L,\rho}$, the function
\begin{equation}\label{Q}
Q[h](y) := \inte e^{\frac{a}{\rho}(y-z)}K(\expyz,1)h(z)\de z\,,
\end{equation}
which is uniformly close to 1 in the interval $(-\infty,L\rho)$, provided that $\rho$ is small enough (see Lemma~\ref{lem:Q}).
We also define
\begin{equation} \label{psi}
\psi_\rho[h](x) = \frac{1}{\rho}\int_{x}^{L\rho} \expay Q[h](y)\de y\,, \qquad x < L\rho\,.
\end{equation}

We consider two different approximations of equation \eqref{equation}, valid in the regions $\{x < L\rho\}$ and $\{x \geq L\rho\}$ respectively.
Precisely, we write
\begin{align}
h(x) &= \frac{1}{\rho} \int_{-\infty}^x \expay Q[h](y)h(y)\de y + R[h](x)\,, & x < L\rho, \label{eq:approx1} \\
h(x) &= \frac{1}{1+\rho}\int_x^\infty h(z)\de z + R[h](x)\,, & x \geq L\rho,  \label{eq:approx2}
\end{align}
where the remainder term $R[h]$ has the expression (recall the normalization $M_\gamma(h)=\frac{1}{1+\rho}$)
\begin{align}
R[h](x) &= - \frac{1}{\rho}\int_{-\infty}^x \de y \int_{-\infty}^{x+\rho\ln(1-e^\frac{y-x}{\rho})}\expaz K(\expyz,1)h(y)h(z)\de z\,, & x < L\rho, \label{eq:remainder1} \\
R[h](x) &= -\frac{1}{\rho} \int_{x}^{\infty}\expay h(y)\de y \int_x^\infty h(z)\de z \nonumber\\
&\qquad + \frac{1}{\rho} \int_{-\infty}^x\de y \int_x^\infty \Bigl( \expaz K(\expyz,1)-\expay \Bigr) h(y)h(z)\de z \nonumber\\
&\qquad + \frac{1}{\rho}\int_{-\infty}^{x}\de y \int_{x+\rho\ln(1-e^\frac{y-x}{\rho})}^x \expaz K(\expyz,1)h(y)h(z)\de z\,, & x \geq L\rho. \label{eq:remainder2}
\end{align}
We can solve explicitly the equations \eqref{eq:approx1}--\eqref{eq:approx2} by treating the term $R$ as a remainder.
This leads to the definition of a map $T:X_{L,\rho}\to C(\R\setminus\{L\rho\})$,
\begin{align}\label{T}
T[h](x):=
\begin{cases}
k_1(h) e^{-\psi_\rho[h](x)} + \widetilde{R}[h](x) & \text{for }x < L\rho,\\
\frac{k_2(h)}{1+\rho} e^{-\frac{1}{1+\rho}(x-L\rho)} + \widetilde{R}[h](x) & \text{for }x \geq L\rho,
\end{cases}
\end{align}
where $k_1(h)$ and $k_2(h)$ are integration constants, depending on $h$, that will be fixed later, and $\widetilde{R}[h]$ is defined as
\begin{align}
\widetilde{R}[h](x) &= R[h](x) - \frac{1}{\rho}\int_x^{L\rho}e^{-(\psi_\rho[h](x)-\psi_\rho[h](y))}\expay Q[h](y)R[h](y)\de y \,, & x < L\rho, \label{eq:remainder3} \\
\widetilde{R}[h](x) &= R[h](x) -\frac{1}{1+\rho}\int_{L\rho}^x e^{\frac{1}{1+\rho}(y-x)}R[h](y)\de y \,, & x \geq L\rho. \label{eq:remainder4}
\end{align}
By construction a fixed point $h\in X_{L,\rho}$ of the map $T$ is indeed a solution to \eqref{eq:approx1}--\eqref{eq:approx2} and, in turn, to \eqref{equation}.
The two arbitrary constants $k_1(h)$, $k_2(h)$ will be chosen in Lemma~\ref{lem:K1K2} in order to satisfy the constraints \eqref{constraints}.

The main result of this section is the following.
\begin{theorem}[Fixed point]\label{thm:fixedpoint}
	There exist $\bar{L}\geq L_0$ and a map $\bar{\rho}:(\bar{L},\infty)\to(0,\rho_0)$ with the following property.
	For every $L>\bar{L}$ and for every $\rho\in(0,\bar{\rho}(L))$ the map $T$ has a unique fixed point in $X_{L,\rho}$, that is, there exists a unique $h\in X_{L,\rho}$ such that $T[h]=h$.
\end{theorem}

\begin{proof}
In view of Lemma~\ref{lem:K1K2}, we can select the two arbitrary constants $k_1(h)$, $k_2(h)$ so that the constraints \eqref{eq:K1K2} are satisfied, for all $L> L_1$ and $\rho\in(0,\rho_1(L))$. In order to show that $T$ maps $X_{L,\rho}$ into itself, it only remains to show the bound
\begin{equation} \label{eq:prooffp}
|T[h](x)| \leq 2e^{\omega(x)}.
\end{equation}
This is a direct consequence of Lemma~\ref{lem:psi}, Lemma~\ref{lem:remainder3-4} and of the first condition in \eqref{eq:K1K2bis}. Indeed, for $x\leq -L\rho$ we have
\begin{align*}
|T[h](x)| \leq \frac32 e^{-\psi_\rho[h](x)} + C\rho e^{(m+a+2)L}e^{\frac{m}{\rho}(x+L\rho)}
\xupref{lem:psi2}{\leq} \Bigl(\frac{3}{2} + C\rho e^{(m+a+2)L} \Bigr) e^{\omega(x)}\,.
\end{align*}
Similarly, using again Lemma~\ref{lem:remainder3-4} and \eqref{eq:K1K2bis}, we have for $x\in(-L\rho,L\rho)$
\begin{align*}
|T[h](x)| \leq \frac32 + C\rho e^{(a+2)L}\,,
\end{align*}
and for $x \geq L\rho$
\begin{align*}
|T[h](x)| \leq \frac{3}{2}e^{-\frac{1}{1+\rho}(x-L\rho)} + C\bigl( e^{-aL}+\rho e^{aL} \bigr)e^{-\frac{1}{2}(x-L\rho)} \leq \Bigl( \frac32 + Ce^{-aL}+C\rho e^{aL} \Bigr) e^{\omega(x)}\,.
\end{align*}
It follows from the previous estimates that the condition \eqref{eq:prooffp} is satisfied, provided we choose $L$ large enough and, in turn, $\rho$ sufficiently small, depending on $L$. More precisely, there exists $L_2\geq L_1$ and a map $\rho_2:(L_2,\infty)\to(0,\rho_0)$ such that for every $L> L_2$ and $\rho\in(0,\rho_2(L))$ one has $T(X_{L,\rho})\subset X_{L,\rho}$.

We finally show that the map $T$ is a contraction in $X_{L,\rho}$:
\begin{equation}\label{eq:prooffp2}
\|T[h_1]-T[h_2]\| \leq \theta \|h_1-h_2\|
\end{equation}
for every $h_1,h_2\in X_{L,\rho}$, for some $\theta<1$.
Arguing as before, we use \eqref{eq:K1K2bis}, \eqref{lem:psi3} (with $y=L\rho$), and \eqref{eq:est-rem3-4bis}, to obtain for $x\leq -L\rho$
\begin{align*}
|T[&h_1](x)-T[h_2](x)| \\
&\leq |k_1(h_1)-k_1(h_2)|e^{-\psi_\rho[h_1](x)} + k_1(h_2)|e^{-\psi_\rho[h_1](x)}-e^{-\psi_\rho[h_2](x)}| + |\widetilde{R}[h_1](x)-\widetilde{R}[h_2](x)| \\
&\leq \Bigl(\frac12 + C\rho e^{(2a+2)L} + C\rho e^{(m+4a+4)L}\Bigr)\|h_1-h_2\| e^{\frac{m}{\rho}(x+L\rho)}\,.
\end{align*}
Similarly we have for $x\in(-L\rho,L\rho)$
\begin{align*}
|T[&h_1](x)-T[h_2](x)| \\
&\leq |k_1(h_1)-k_1(h_2)|e^{-\psi_\rho[h_1](x)} + k_1(h_2)|e^{-\psi_\rho[h_1](x)}-e^{-\psi_\rho[h_2](x)}| + |\widetilde{R}[h_1](x)-\widetilde{R}[h_2](x)| \\
&\leq \Bigl(\frac12 + C\rho e^{(3a+2)L} + C\rho e^{(5a+4)L}\Bigr)\|h_1-h_2\|\,,
\end{align*}
and for $x\geq L\rho$
\begin{align*}
|T[h_1](x)-T[h_2](x)|
&\leq |k_2(h_1)-k_2(h_2)|e^{-\frac{1}{1+\rho}(x-L\rho)} + |\widetilde{R}[h_1](x)-\widetilde{R}[h_2](x)| \\
&\leq \Bigl(\frac12 + Ce^{-aL} +C\rho e^{aL} \Bigr)\|h_1-h_2\| e^{-\frac12(x-L\rho)}\,.
\end{align*}
Hence, by possibly taking a larger $L$ and, in turn, a smaller $\rho$, we obtain that the condition \eqref{eq:prooffp2} is satisfied.
The conclusion of the theorem follows now from Banach Fixed Point Theorem.
\end{proof}

\begin{remark}[Continuity] \label{rm:continuity}
Notice that, even if the function $T[h]$ defined in \eqref{T} might be in principle discontinuous at the point $L\rho$, the fixed point constructed in Theorem~\ref{thm:fixedpoint} is automatically continuous on $\R$: indeed, the right-hand side of \eqref{equation} is continuous by Lebesgue's Dominated Convergence Theorem.
\end{remark}

\begin{remark}[Uniqueness] \label{rm:uniq}
A consequence of Theorem~\ref{thm:fixedpoint} is the following weak form of uniqueness of the solution to \eqref{equation}: for a given $L>\bar{L}$ and for $\rho<\bar{\rho}(L)$, the fixed point $h\in X_{L,\rho}$ coincides with the fixed point in $X_{L',\rho}$ for every $L'<L$. Indeed, from definition \eqref{space} of the space it is clear that $h\in X_{L',\rho}$, so that the claim follows by the uniqueness of the fixed point in $X_{L',\rho}$.
The auxiliary parameter $L$ has been introduced to take into account the transition layer in which the solution has a strong variation.
\end{remark}


\section{Technical estimates} \label{sect:estimates}

We collect in this section the technical estimates needed in the proof of the existence of a fixed point in Theorem~\ref{thm:fixedpoint}.
Recall that $C$ always denotes a generic constant, depending only on the kernel $K$ and on $m$, but not on $L$ and $\rho$, which may change from line to line.
We start with a lemma showing uniform integral estimates satisfied by any function $h\in X_{L,\rho}$.

\begin{lemma}\label{lem:prelim}
	For every $h\in X_{L,\rho}$ one has
	\begin{equation}\label{eq:prelim1}
	\frac{1}{\rho}\inte \expay|h(y)|\de y \leq C e^{aL}\,,
	\end{equation}
	\begin{equation}\label{eq:prelim2}
	\frac{1}{\rho}\inte e^{\frac{\delta-a}{\rho}y}|h(y)|\de y \leq Ce^{aL}\,,
	\end{equation}
	where $\delta$ is the parameter appearing in \eqref{kernel3}.
\end{lemma}

\begin{proof}
	The estimates in the statement follow by inserting the bound $|h(x)|\leq 2e^{\omega(x)}$ in the integrals, see \eqref{parameters3}, and by elementary computations.
\end{proof}

In the next three lemmas we collect some properties of the functions $Q[h]$ and $\psi_\rho[h]$, defined in \eqref{Q} and \eqref{psi} respectively, which will be instrumental in the following.

\begin{lemma}\label{lem:Q}
	Let $Q[h]$ be the function defined in \eqref{Q}, for $h\in X_{L,\rho}$.
	There exists a constant $c_0>0$, depending only on $K_0$, $a$, $\delta$, and $m$, such that if $\rho e^{(2a+2)L}<c_0$ then
	\begin{equation}\label{lem:Q2}
	|Q[h](y)-1| \leq \frac12 \qquad\text{for }y\in(-\infty,L\rho)\,,
	\end{equation}
	for every $h\in X_{L,\rho}$.
\end{lemma}

\begin{proof}
The lemma is a consequence of the constraint $M(h)=1$, see \eqref{constraints}. We consider first the case $y\in(-\infty,-L\rho)$: using \eqref{kernel3}--\eqref{kernel4}
\begin{align} \label{lem:Q2bis}
|Q[h](y)-1| &= \bigg| \inte e^{\frac{a}{\rho}(y-z)}K(\expyz,1)h(z)\de z - \inte h(z)\de z\bigg| \nonumber\\
& \leq C\int_{-\infty}^y e^{\frac{a+1}{\rho}(y-z)}|h(z)|\de z + C \int_{y}^\infty e^{\frac{\delta}{\rho}(y-z)}|h(z)|\de z + \int_{-\infty}^y |h(z)|\de z \\
& \leq C\int_{-\infty}^y e^{\frac{a+1}{\rho}(y-z)}e^{\frac{m}{\rho}(z+L\rho)}\de z + C \int_{y}^\infty e^{\frac{\delta}{\rho}(y-z)}\de z + C \int_{-\infty}^y e^{\frac{m}{\rho}(z+L\rho)}\de z \leq C\rho\,.\nonumber
\end{align}
We next consider the case $y\in(-L\rho,L\rho)$: arguing as before, we have
\begin{align} \label{lem:Q2ter}
|Q[h](y)-1|
& \leq C\int_{-\infty}^y e^{\frac{a+1}{\rho}(y-z)}e^{\omega(z)}\de z + C \int_{y}^\infty e^{\frac{\delta}{\rho}(y-z)}\de z + C \int_{-\infty}^y e^{\omega(z)}\de z \nonumber\\
& \leq C\rho \bigl(1 + L + e^{(2a+2)L} \bigr)\,.
\end{align}
Since the constant $C$ in \eqref{lem:Q2bis}--\eqref{lem:Q2ter} depends only on $K_0$, $a$, $\delta$, and $m$, it is clear that \eqref{lem:Q2} holds provided that the quantity $\rho e^{(2a+2)L}$ is small enough.
\end{proof}

\begin{remark}
Notice that throughout the paper we always assume that the parameters $L$ and $\rho$ fulfill condition \eqref{parameters2bis}.
Hence the bound \eqref{lem:Q2} is satisfied and, in particular, $Q[h]$ is a positive function.
It also follows from \eqref{lem:Q2bis}--\eqref{lem:Q2ter} that the function $Q[h]$ is uniformly close to the constant 1, provided $\rho$ is small enough.
\end{remark}

\begin{lemma}\label{lem:Q3}
For every $h_1,h_2\in X_{L,\rho}$ and $y<L\rho$ one has
	\begin{equation}\label{lem:Q4}
	\big|Q[h_1](y)-Q[h_2](y)\big|\leq C\rho e^{(2a+2)L}\|h_1-h_2\|\,.
	\end{equation}
\end{lemma}

\begin{proof}
The proof of \eqref{lem:Q4} follows by estimates entirely similar to the ones in the proof of Lemma~\ref{lem:Q}, using the inequality $|h_1(z)-h_2(z)|\leq \|h_1-h_2\|e^{\omega(z)}$. We omit the details here.
\end{proof}

\begin{lemma}\label{lem:psi}
	Let $\psi_\rho[h]$ be the function defined in \eqref{psi}, for $h\in X_{L,\rho}$. Then the map
	\begin{equation}\label{lem:psi2}
	x\mapsto e^{-\psi_\rho[h](x)-\frac{2m}{\rho}x} \qquad\text{is monotone non-decreasing for }x\in(-\infty,-L\rho).
	\end{equation}
	In turn, also the maps $x\mapsto e^{-\frac12\psi_\rho[h](x)-\frac{a}{\rho}x}$, $x\mapsto e^{-\frac12\psi_\rho[h](x)-\frac{1}{\rho}x}$ are monotone non-decreasing in the same interval.
	Furthermore, one has for every $h_1,h_2\in X_{L,\rho}$ and for $x<y\leq L\rho$
	\begin{align} \label{lem:psi3}
	\Big| e^{\psi_\rho[h_1](y)-\psi_\rho[h_1](x)} - & e^{\psi_\rho[h_2](y)-\psi_\rho[h_2](x)} \Big| \nonumber\\
	&\leq 
	\begin{cases}
	C\rho e^{(2a+2)L}\|h_1-h_2\| e^{-\frac{a}{\rho}x}e^{\frac{2m}{\rho}(x-y)} & x\leq -L\rho,\, y\leq -L\rho, \\
	C\rho e^{(2a+2)L}\|h_1-h_2\| e^{-\frac{a}{\rho}x}e^{\frac{2m}{\rho}(x+L\rho)} & x\leq -L\rho,\, y\in[-L\rho,L\rho],\\
	C\rho e^{(3a+2)L}\|h_1-h_2\| & x\in(-L\rho,L\rho).
	\end{cases}
	\end{align}
\end{lemma}

\begin{proof}
The statement \eqref{lem:psi2} is true provided $\frac{\de}{\de x}\psi_\rho[h](x)+\frac{2m}{\rho}\leq 0$.
Using Lemma~\ref{lem:Q} we have for $x<-L\rho$
$$
\frac{\de}{\de x}(\psi_\rho[h](x)) = - \textstyle\frac{1}{\rho} e^{-\frac{a}{\rho}x}Q[h](x) \leq -\frac{1}{2\rho}e^{-\frac{a}{\rho}x} \leq -\frac{1}{2\rho}e^{aL} \leq -\frac{2m}{\rho}
$$
by the choice of $L\geq L_0\geq\frac{1}{a}\ln(4m)$ in \eqref{parameters}. The monotonicity of the other two maps in the statement is an easy consequence of the same inequality.

We now prove \eqref{lem:psi3}. We first observe that by \eqref{lem:Q4}
\begin{align} \label{lem:psi4}
\Big| e^{\psi_\rho[h_1](y)-\psi_\rho[h_1](x)} &- e^{\psi_\rho[h_2](y)-\psi_\rho[h_2](x)} \Big| \nonumber\\
& \leq \max_{i=1,2} \bigl( e^{\psi_\rho[h_i](y)-\psi_\rho[h_i](x)} \bigr) \frac{1}{\rho}\int_x^y \expaz \big| Q[h_1](z)-Q[h_2](z)\big| \de z \nonumber\\
& \leq Ce^{(2a+2)L}\|h_1-h_2\| \max_{i=1,2} \bigl( e^{\psi_\rho[h_i](y)-\psi_\rho[h_i](x)} \bigr) \int_x^y \expaz\de z \nonumber\\
& \leq C\rho e^{(2a+2)L}\|h_1-h_2\| e^{-\frac{a}{\rho}x} \max_{i=1,2} \bigl( e^{\psi_\rho[h_i](y)-\psi_\rho[h_i](x)} \bigr)\,.
\end{align}
We immediately obtain the estimate \eqref{lem:psi3} in the region $x\in(-L\rho,L\rho)$ simply by observing that the maximum in \eqref{lem:psi4} is actually bounded by the constant 1, since $x<y$ and $Q[h]\geq0$ by \eqref{lem:Q2}.
For $x\leq -L\rho$ we have instead, by using the monotonicity property \eqref{lem:psi2},
\begin{align*}
\max_{i=1,2} \bigl( e^{\psi_\rho[h_i](y)-\psi_\rho[h_i](x)} \bigr) \leq e^{\frac{2m}{\rho}(x-y)}
\end{align*}
for $x<y\leq -L\rho$, while
\begin{align*}
\max_{i=1,2} \bigl( e^{\psi_\rho[h_i](y)-\psi_\rho[h_i](x)} \bigr)
\leq e^{\frac{2m}{\rho}(x+L\rho)} \max_{i=1,2} \bigl( e^{\psi_\rho[h_i](y)-\psi_\rho[h_i](-L\rho)} \bigr)
\leq e^{\frac{2m}{\rho}(x+L\rho)} 
\end{align*}
for $y\in(-L\rho,L\rho)$. By inserting these inequalities in \eqref{lem:psi4} we conclude that \eqref{lem:psi3} holds.
\end{proof}

We now obtain explicit bounds on the remainders $R[h]$ and $\widetilde{R}[h]$.

\begin{lemma}\label{lem:remainder1-2}
	Let $R[h]$ be the remainder term defined in \eqref{eq:remainder1}--\eqref{eq:remainder2}, for $h\in X_{L,\rho}$.
	The following estimates hold:
	\begin{align} \label{eq:est-rem1-2}
	|R[h](x)| \leq
	\begin{cases}
	C\rho e^{-\frac{a}{\rho}x} e^{\frac{2m}{\rho}(x+L\rho)} & \text{for } x \leq -L\rho,\\
	C\rho e^{(a+2)L} & \text{for } x\in(-L\rho,L\rho),\\
	C \bigl( e^{-aL} + \rho e^{aL}\bigr)e^{-\frac12(x-L\rho)} & \text{for }x \geq L\rho,
	\end{cases}
	\end{align}
	\begin{align} \label{eq:est-rem1-2bis}
	|R[h_1](x)-R[h_2](x)| \leq
	\begin{cases}
	C\rho \|h_1-h_2\| e^{-\frac{a}{\rho}x} e^{\frac{2m}{\rho}(x+L\rho)} & \text{for } x \leq -L\rho,\\
	C\rho \|h_1-h_2\|e^{(a+2)L} & \text{for }x\in(-L\rho,L\rho),\\
	C \bigl( e^{-aL} + \rho e^{aL}\bigr)\|h_1-h_2\|e^{-\frac12(x-L\rho)} & \text{for }x \geq L\rho,
	\end{cases}
	\end{align}
	for every $h,h_1,h_2\in X_{L,\rho}$.
\end{lemma}

\begin{proof}
We estimate $R[h]$ separately in the three regions $(-\infty,-L\rho]$, $(-L\rho,L\rho)$, $[L\rho,\infty)$.
We just give the proof of \eqref{eq:est-rem1-2}, since the bound \eqref{eq:est-rem1-2bis} on $|R[h_1](x)-R[h_2](x)|$ follows by completely analogous arguments.

\medskip
\noindent\textit{Case 1: $x \leq -L\rho$.}
Using the expression in \eqref{eq:remainder1}, the properties of the kernel \eqref{kernel3}--\eqref{kernel4}, and the bound $|h(\xi)|\leq 2e^{\omega(\xi)}$ for $\omega$ as in \eqref{parameters3}, we have
\begin{align*}
|R[h](x)| &\leq \frac{1}{\rho}\int_{-\infty}^x \de y \int_{-\infty}^x \expaz K(\expyz,1)|h(y)||h(z)|\de z \\
& \leq \frac{4(K_0+1)}{\rho}\int_{-\infty}^x \de y \int_{-\infty}^y \expaz\expyz e^{\frac{m}{\rho}(y+L\rho)}e^{\frac{m}{\rho}(z+L\rho)}\de z \\
&\qquad+ \frac{4(K_0+1)}{\rho}\int_{-\infty}^x \de y \int_{y}^x \expay e^{\frac{m}{\rho}(y+L\rho)}e^{\frac{m}{\rho}(z+L\rho)}\de z \\
& \leq C\rho e^{-\frac{a}{\rho}x}e^{\frac{2m}{\rho}(x+L\rho)}\,,
\end{align*}
the last estimate following by computing explicitly the integrals (recall that $m>a+1$).

\medskip
\noindent\textit{Case 2: $-L\rho < x<L\rho$.}
Using as before \eqref{kernel3}--\eqref{kernel4} and $|h(\xi)|\leq 2e^{\omega(\xi)}$, we have
\begin{align*}
|R[h](x)| &\leq \frac{1}{\rho}\int_{-\infty}^x \de y \int_{-\infty}^x \expaz K(\expyz,1)|h(y)||h(z)|\de z \\
& \leq \frac{C}{\rho}\int_{-\infty}^x \de y \int_{-\infty}^y \expaz\expyz e^{\omega(y)}e^{\omega(z)}\de z + \frac{C}{\rho}\int_{-\infty}^x \de y \int_{y}^x \expay e^{\omega(y)}e^{\omega(z)}\de z\,.
\end{align*}
It now follows by an elementary computation that the last integrals are bounded by $C\rho e^{(a+2)L}$.

\medskip
\noindent\textit{Case 3: $x \geq L\rho$.} We estimate separately the three terms appearing in the expression \eqref{eq:remainder2} for $R[h]$. For the first integral, we have
\begin{align*}
|R_1[h](x)| &\leq \frac{1}{\rho}\int_x^\infty \expay|h(y)|\de y \int_x^\infty |h(z)|\de z \\
& \leq \frac{4}{\rho} \int_x^\infty \expay e^{-\frac{1}{2}(y-L\rho)}\de y \int_x^\infty e^{-\frac12(z-L\rho)}\de z \leq C e^{-aL}e^{-(x-L\rho)}\,.
\end{align*}
For the second term, we have using \eqref{kernel3}
\begin{align*}
|R_2[h](x)| & \leq \frac{1}{\rho}\int_{-\infty}^x \de y \int_x^\infty \Big| \expaz K(\expyz,1)-\expay\Big| |h(y)||h(z)|\de z \\
& \leq \frac{K_0}{\rho}\int_{-\infty}^x \de y \int_x^\infty \expay e^{\frac{\delta}{\rho}(y-z)}|h(y)||h(z)|\de z \\
& \leq \frac{2K_0}{\rho} \int_{-\infty}^x e^{\frac{\delta-a}{\rho}y}|h(y)|\de y \int_{x}^\infty e^{-\frac{\delta}{\rho}z}e^{-\frac12(z-L\rho)}\de z
\leq C\rho e^{aL}e^{-\frac12(x-L\rho)}\,,
\end{align*}
where we used \eqref{eq:prelim2} in the last inequality.
Finally, the last integral in \eqref{eq:remainder2} is over the region $\Omega_\rho$, which we split into the two subregions $A_\rho$, $B_\rho$, see \eqref{eq:regions}, as explained at the beginning of Section~\ref{sect:fixedpoint}: we have, using \eqref{kernelA}--\eqref{kernelB},
\begin{align*}
|R_3[h](x)| & \leq \frac{1}{\rho}\iint_{\Omega_\rho} \expaz K(\expyz,1)|h(y)||h(z)|\de y\de z \\
& \leq \frac{C}{\rho}\iint_{A_\rho} \expay|h(y)||h(z)|\de y \de z + \frac{C}{\rho}\iint_{B_\rho}\expaz\expyz|h(y)||h(z)|\de y\de z\,.
\end{align*}
We estimate the two terms separately: using that $|h(z)|\leq Ce^{-\frac12(x-L\rho)}$ for $z\in(x-\rho\ln2,x)$, we obtain
\begin{align*}
\frac{1}{\rho}\iint_{A_\rho} \expay|h(y)||h(z)|\de y \de z
& \leq \frac{C}{\rho}e^{-\frac12(x-L\rho)}\int_{x-\rho\ln2}^x \de z \int_{x+\rho\ln(1-e^{\frac{z-x}{\rho}})}^x \expay |h(y)|\de y \\
& \leq Ce^{-\frac12(x-L\rho)}\int_{-\infty}^\infty \expay|h(y)|\de y
\leq C\rho e^{aL} e^{-\frac12(x-L\rho)}\,,
\end{align*}
where we used also \eqref{eq:prelim1} in the last inequality.
Observe now that for $(y,z)\in B_\rho$ one has $y\in(x-\rho\ln2,x)$ and, in turn, $|h(y)|\leq Ce^{-\frac12(x-L\rho)}$.
Hence for the integral in the region $B_\rho$ we have
\begin{align*}
\frac{1}{\rho}\iint_{B_\rho}\expaz\expyz & |h(y)||h(z)|\de y\de z \\
& \leq \frac{C}{\rho}e^{-\frac12(x-L\rho)} \int_{-\infty}^{x-\rho\ln2} e^{-\frac{a+1}{\rho}z}|h(z)|\int_{x+\rho\ln(1-e^{\frac{z-x}{\rho}})}^x e^{\frac{y}{\rho}}\de y \de z\\
& = Ce^{-\frac12(x-L\rho)} \int_{-\infty}^{x-\rho\ln2}\expaz|h(z)|\de z
\leq C\rho e^{aL} e^{-\frac12(x-L\rho)}\,,
\end{align*}
the last inequality following from \eqref{eq:prelim1}.
Collecting all the previous estimates, we obtain the bound for $R[h]=R_1[h]+R_2[h]+R_3[h]$ in the region $\{x\geq L\rho\}$.
\end{proof}

In the computations leading to the result in the following lemma, it is often useful to bear in mind the identity
\begin{equation} \label{eq:derivativepsi}
\frac{\de}{\de\zeta}(\psi_\rho[h](\zeta)) = -\textstyle\frac{1}{\rho}e^{-\frac{a}{\rho}\zeta} Q[h](\zeta)\,.
\end{equation}
This allows, for instance, to compute explicitly the integral (for $\xi<\eta \leq L\rho$)
\begin{equation} \label{eq:integralpsi}
\frac{1}{\rho}\int_{\xi}^\eta e^{-\frac{a}{\rho}\zeta}Q[h](\zeta) e^{\psi_\rho[h](\zeta)}\de \zeta = e^{\psi_\rho[h](\xi)}- e^{\psi_\rho[h](\eta)}\,.
\end{equation}

\begin{lemma}\label{lem:remainder3-4}
	Let $\widetilde{R}[h]$ be the remainder term defined in \eqref{eq:remainder3}--\eqref{eq:remainder4}, for $h\in X_{L,\rho}$.
	The following estimates hold:
	\begin{align} \label{eq:est-rem3-4}
	|\widetilde{R}[h](x)| \leq
	\begin{cases}
	C\rho e^{(m+a+2)L} e^{\frac{m}{\rho}(x+L\rho)} & \text{for } x \leq -L\rho,\\
	C\rho e^{(a+2)L} & \text{for } x\in(-L\rho,L\rho),\\
	C \bigl( e^{-aL} + \rho e^{aL}\bigr)e^{-\frac12(x-L\rho)} & \text{for }x \geq L\rho,
	\end{cases}
	\end{align}
	\begin{align} \label{eq:est-rem3-4bis}
	|\widetilde{R}[h_1](x)-\widetilde{R}[h_2](x)| \leq
	\begin{cases}
	C\rho \|h_1-h_2\| e^{(m+4a+4)L} e^{\frac{m}{\rho}(x+L\rho)} & \text{for } x \leq -L\rho,\\
	C\rho \|h_1-h_2\|e^{(5a+4)L} & \text{for }x\in(-L\rho,L\rho),\\
	C \bigl( e^{-aL} + \rho e^{aL}\bigr)\|h_1-h_2\|e^{-\frac12(x-L\rho)} & \text{for }x \geq L\rho,
	\end{cases}
	\end{align}
	for every $h,h_1,h_2\in X_{L,\rho}$.
\end{lemma}

\begin{proof}
We use Lemma~\ref{lem:remainder1-2} to estimate the integral remainder in the definition of $\widetilde{R}[h]$.
As before, we proceed separately in the three regions $(-\infty,-L\rho]$, $(-L\rho,L\rho)$, $[L\rho,\infty)$.

\medskip
\noindent\textit{Case 1: $-L\rho < x < L\rho$.}
To prove \eqref{eq:est-rem3-4}, we find an upper bound for the term
\begin{align*}
\bigg| \frac{1}{\rho}\int_x^{L\rho}e^{-(\psi_\rho[h](x)-\psi_\rho[h](y))} & \expay Q[h](y)R[h](y)\de y \bigg| \\
& \xupref{eq:est-rem1-2}{\leq} C e^{(a+2)L} \int_x^{L\rho} e^{-(\psi_\rho[h](x)-\psi_\rho[h](y))} \expay Q[h](y) \de y \\
& = C\rho e^{(a+2)L} \Bigl( 1 - e^{-\psi_\rho[h](x)} \Bigr) \leq C\rho e^{(a+2)L}\,,
\end{align*}
where we computed explicitly the integral, according to \eqref{eq:integralpsi}.

To obtain the estimate \eqref{eq:est-rem3-4bis} in the same region, we have to bound the difference
\begin{align} \label{tec0}
\frac{1}{\rho}&\int_x^{L\rho} \Big| e^{\psi_\rho[h_1](y)-\psi_\rho[h_1](x)} Q[h_1](y)R[h_1](y)
- e^{\psi_\rho[h_2](y)-\psi_\rho[h_2](x)} Q[h_2](y)R[h_2](y)\Big| \expay\de y \nonumber \\
& \leq \frac{1}{\rho}\int_x^{L\rho} \Big| e^{\psi_\rho[h_1](y)-\psi_\rho[h_1](x)} - e^{\psi_\rho[h_2](y)-\psi_\rho[h_2](x)} \Big| Q[h_1](y) |R[h_1](y)| \expay\de y \\
& \qquad + \frac{1}{\rho}\int_x^{L\rho} e^{\psi_\rho[h_2](y)-\psi_\rho[h_2](x)} \big|Q[h_1](y)-Q[h_2](y) \big| |R[h_1](y)| \expay\de y \nonumber \\
& \qquad + \frac{1}{\rho}\int_x^{L\rho} e^{\psi_\rho[h_2](y)-\psi_\rho[h_2](x)} Q[h_2](y) \big| R[h_1](y) - R[h_2](y) \big| \expay\de y
=: I_1 + I_2 + I_3\,.\nonumber
\end{align}
For the first term $I_1$ in \eqref{tec0} we have, using \eqref{lem:Q2}, \eqref{lem:psi3}, and \eqref{eq:est-rem1-2},
\begin{align*}
I_1 \leq C\rho e^{(4a+4)L}\|h_1-h_2\| \int_x^{L\rho} \expay\de y \leq C\rho^2 e^{(5a+4)L}\|h_1-h_2\|\,.
\end{align*}
To obtain a bound on $I_2$, we use Lemma~\ref{lem:Q3}, \eqref{eq:est-rem1-2}, and the uniform bound $Q[h_2](y)\geq\frac12$:
\begin{align*}
I_2 &\leq C\rho e^{(3a+4)L}\|h_1-h_2\| \int_x^{L\rho} e^{\psi_\rho[h_2](y)-\psi_\rho[h_2](x)} \expay Q[h_2](y)\de y
\xupref{eq:integralpsi}{\leq} C\rho^2 e^{(3a+4)L}\|h_1-h_2\|\,.
\end{align*}
Finally we use \eqref{eq:est-rem1-2bis} to estimate $I_3$:
\begin{align*}
I_3 \leq C e^{(a+2)L}\|h_1-h_2\| \int_x^{L\rho} e^{\psi_\rho[h_2](y)-\psi_\rho[h_2](x)} \expay Q[h_2](y) \de y
\xupref{eq:integralpsi}{\leq} C \rho e^{(a+2)L}\|h_1-h_2\| \,.
\end{align*}
Inserting the previous estimates in \eqref{tec0}, and recalling also \eqref{eq:est-rem1-2bis}, we obtain the desired bound \eqref{eq:est-rem3-4bis} in the region $(-L\rho,L\rho)$.

\medskip
\noindent\textit{Case 2: $x \leq -L\rho$.}
In this case we split the integral remainder into two parts:
\begin{align} \label{tec1}
\bigg| \frac{1}{\rho}\int_x^{L\rho}e^{-(\psi_\rho[h](x)-\psi_\rho[h](y))} & \expay Q[h](y)R[h](y)\de y \bigg| \nonumber\\
& \xupref{eq:est-rem1-2}{\leq} C e^{(a+2)L} \int_{-L\rho}^{L\rho}e^{-(\psi_\rho[h](x)-\psi_\rho[h](y))} \expay Q[h](y) \de y \\
& \qquad + C\int_x^{-L\rho}e^{-(\psi_\rho[h](x)-\psi_\rho[h](y))} \expay Q[h](y) \expay e^{\frac{2m}{\rho}(y+L\rho)} \de y\,. \nonumber
\end{align}
For the first term in \eqref{tec1} we have as in the previous step
\begin{align*}
C e^{(a+2)L} \int_{-L\rho}^{L\rho}e^{-(\psi_\rho[h](x)-\psi_\rho[h](y))} \expay Q[h](y) \de y
& \xupref{eq:integralpsi}{=} C\rho e^{(a+2)L}e^{-\psi_\rho[h](x)} \Bigl( e^{\psi_\rho[h](-L\rho)}-1 \Bigr) \\
& \leq C\rho e^{(a+2)L}e^{\frac{2m}{\rho}(x+L\rho)}\,,
\end{align*}
where the last estimate follows from Lemma~\ref{lem:psi}.
For the second term in \eqref{tec1} we have instead, using again the monotonicity property \eqref{lem:psi2},
\begin{align*}
\int_x^{-L\rho}e^{-(\psi_\rho[h](x)-\psi_\rho[h](y))} & Q[h](y) e^{\frac{2m}{\rho}(y+L\rho)} e^{-\frac{2a}{\rho}y} \de y \\
& \xupref{lem:Q2}{\leq} 2 e^{2mL}e^{-\psi_\rho[h](x)} \int_x^{-L\rho} e^{\psi_\rho[h](y)+\frac{2m}{\rho}y}e^{-\frac{2a}{\rho}y}\de y \\
& \leq 2 e^{\frac{2m}{\rho}(x+L\rho)}\int_x^{-L\rho} e^{-\frac{2a}{\rho}y}\de y \\
& \leq C\rho e^{\frac{2m}{\rho}(x+L\rho)}e^{-\frac{2a}{\rho}x}
\leq C\rho e^{mL} e^{\frac{m}{\rho}(x+L\rho)}
\end{align*}
(for the last estimate, recall that $m>2a$).
By plugging the previous inequalities into \eqref{tec1} we obtain the bound \eqref{eq:est-rem3-4} in the region $(-\infty,-L\rho]$.

To prove the estimate \eqref{eq:est-rem3-4bis} in the same region, we split the integral as in \eqref{tec1}:
\begin{align} \label{tec2}
\frac{1}{\rho}\int_x^{L\rho} \Big| e^{\psi_\rho[h_1](y)-\psi_\rho[h_1](x)} Q[h_1](y)R[h_1](y)
- e^{\psi_\rho[h_2](y)-\psi_\rho[h_2](x)} & Q[h_2](y) R[h_2](y)\Big| \expay\de y \nonumber\\
& = J_1 + J_2\,,
\end{align}
where $J_1$ is the integral from $-L\rho$ to $L\rho$ and $J_2$ is the integral over $(x,-L\rho)$.
We first consider the term $J_1$, and we argue as in \eqref{tec0}:
\begin{align*}
J_1
& \leq \frac{1}{\rho}\int_{-L\rho}^{L\rho} \Big| e^{\psi_\rho[h_1](y)-\psi_\rho[h_1](x)} - e^{\psi_\rho[h_2](y)-\psi_\rho[h_2](x)} \Big| Q[h_1](y) |R[h_1](y)| \expay\de y \\
& \qquad + \frac{1}{\rho}\int_{-L\rho}^{L\rho} e^{\psi_\rho[h_2](y)-\psi_\rho[h_2](x)} \big|Q[h_1](y)-Q[h_2](y) \big| |R[h_1](y)| \expay\de y \nonumber \\
& \qquad + \frac{1}{\rho}\int_{-L\rho}^{L\rho} e^{\psi_\rho[h_2](y)-\psi_\rho[h_2](x)} Q[h_2](y) \big| R[h_1](y) - R[h_2](y) \big| \expay\de y \nonumber\\
&=: J_{1,1} + J_{1,2}+ J_{1,3}\,. \nonumber
\end{align*}
For the first integral we have by \eqref{lem:psi3}
\begin{align*}
J_{1,1} & \leq C e^{(2a+2)L}\|h_1-h_2\| e^{-\frac{a}{\rho}x}e^{\frac{2m}{\rho}(x+L\rho)} \int_{-L\rho}^{L\rho}|R[h_1](y)|\expay\de y \\
& \xupref{eq:est-rem1-2}{\leq} C\rho^2 e^{(m+4a+4)L}\|h_1-h_2\| e^{\frac{m}{\rho}(x+L\rho)}\,.
\end{align*}
For the term $J_{1,2}$ we use \eqref{lem:Q4} and \eqref{eq:est-rem1-2}, together with the bound $Q[h_2](y)\geq\frac12$:
\begin{align*}
J_{1,2} &\leq C \rho e^{(3a+4)L}\|h_1-h_2\| \int_{-L\rho}^{L\rho} e^{\psi_\rho[h_2](y)-\psi_\rho[h_2](x)} \expay Q[h_2](y)\de y \\
& \xupref{eq:integralpsi}{\leq} C \rho^2 e^{(3a+4)L}\|h_1-h_2\| e^{-\psi_\rho[h_2](x)}e^{\psi_\rho[h_2](-L\rho)} \\
& \leq C \rho^2 e^{(3a+4)L}\|h_1-h_2\| e^{\frac{2m}{\rho}(x+L\rho)}\,,
\end{align*}
the last inequality following from Lemma~\ref{lem:psi}.
For the term $J_{1,3}$ we use \eqref{eq:est-rem1-2bis}:
\begin{align*}
J_{1,3} &\leq C e^{(a+2)L}\|h_1-h_2\| \int_{-L\rho}^{L\rho} e^{\psi_\rho[h_2](y)-\psi_\rho[h_2](x)} \expay Q[h_2](y) \de y \\
& \xupref{eq:integralpsi}{\leq} C \rho e^{(a+2)L}\|h_1-h_2\| e^{\frac{2m}{\rho}(x+L\rho)}\,.
\end{align*}
This completes the estimate of $J_1$ in \eqref{tec2}.
It remains to consider the term $J_2$, and we proceed similarly:
\begin{align*}
J_2 & \leq \frac{1}{\rho}\int_x^{-L\rho} \Big| e^{\psi_\rho[h_1](y)-\psi_\rho[h_1](x)} - e^{\psi_\rho[h_2](y)-\psi_\rho[h_2](x)} \Big| Q[h_1](y) |R[h_1](y)| \expay\de y \\
& \qquad + \frac{1}{\rho}\int_x^{-L\rho} e^{\psi_\rho[h_2](y)-\psi_\rho[h_2](x)} \big|Q[h_1](y)-Q[h_2](y) \big| |R[h_1](y)| \expay\de y \nonumber \\
& \qquad + \frac{1}{\rho}\int_x^{-L\rho} e^{\psi_\rho[h_2](y)-\psi_\rho[h_2](x)} Q[h_2](y) \big| R[h_1](y) - R[h_2](y) \big| \expay\de y \nonumber\\
&=: J_{2,1} + J_{2,2} + J_{2,3}\,. \nonumber
\end{align*}
For the first integral we have by \eqref{lem:psi3}
\begin{align*}
J_{2,1} & \leq C e^{(2a+2)L}\|h_1-h_2\| e^{\frac{2m-a}{\rho}x} \int_{x}^{-L\rho}|R[h_1](y)|e^{-\frac{2m}{\rho}y}\expay\de y \\
& \xupref{eq:est-rem1-2}{\leq} C\rho e^{(2m+2a+2)L}\|h_1-h_2\| e^{\frac{2m-a}{\rho}x}  \int_{x}^{-L\rho}e^{-\frac{2a}{\rho}y}\de y \\
& \leq C\rho^2 e^{(2m+2a+2)L}\|h_1-h_2\| e^{\frac{2m-3a}{\rho}x}
\leq C\rho^2 e^{(m+2a+2)L}\|h_1-h_2\| e^{\frac{m}{\rho}(x+L\rho)}
\end{align*}
(recall that $m>3a$).
For the term $J_{2,2}$ we use \eqref{lem:Q4} and \eqref{eq:est-rem1-2}:
\begin{align*}
J_{2,2} &\leq C\rho e^{(2a+2)L}\|h_1-h_2\| \int_{x}^{-L\rho} e^{\psi_\rho[h_2](y)-\psi_\rho[h_2](x)} e^{\frac{2m}{\rho}(y+L\rho)}e^{-\frac{2a}{\rho}y}\de y \\
& \leq C\rho e^{(2a+2)L}\|h_1-h_2\| e^{\frac{2m}{\rho}(x+L\rho)} \int_x^{-L\rho} e^{-\frac{2a}{\rho}y}\de y \\
& \leq C\rho^2 e^{(2a+2)L}\|h_1-h_2\| e^{\frac{2m}{\rho}(x+L\rho)} e^{-\frac{2a}{\rho}x}
\leq C\rho^2 e^{(m+2a+2)L}\|h_1-h_2\| e^{\frac{m}{\rho}(x+L\rho)}\,,
\end{align*}
the second inequality following from the monotonicity in Lemma~\ref{lem:psi}, and the last one since $m>2a$.
For the term $J_{2,3}$ we use \eqref{eq:est-rem1-2bis} and we conclude as in the previous estimate:
\begin{align*}
J_{2,3} &\leq C \|h_1-h_2\| \int_{x}^{-L\rho} e^{\psi_\rho[h_2](y)-\psi_\rho[h_2](x)} e^{\frac{2m}{\rho}(y+L\rho)}e^{-\frac{2a}{\rho}y} \de y \\
& \leq C\rho e^{mL}\|h_1-h_2\| e^{\frac{m}{\rho}(x+L\rho)}\,.
\end{align*}
This completes the estimate of $J_2$ in \eqref{tec2}.
By collecting all the previous inequalities and inserting them in \eqref{tec2}, the bound \eqref{eq:est-rem3-4bis} in the region $(-\infty,-L\rho]$ is proved.

\medskip
\noindent\textit{Case 3: $x \geq L\rho$.}
In this case we have to estimate
\begin{align*}
\bigg|\frac{1}{1+\rho}\int_{L\rho}^x e^{\frac{1}{1+\rho}(y-x)}R[h](y)\de y \bigg|
& \xupref{eq:est-rem1-2}{\leq} C(e^{-aL}+\rho e^{aL}) \int_{L\rho}^x e^{\frac{1}{1+\rho}(y-x)}e^{-\frac12(y-L\rho)}\de y \\
& \leq C(e^{-aL}+\rho e^{aL})e^{-\frac12(x-L\rho)}\,,
\end{align*}
which gives the last bound in \eqref{eq:est-rem3-4}. Similarly, \eqref{eq:est-rem3-4bis} follows using \eqref{eq:est-rem1-2bis}.
\end{proof}

We conclude this section by showing that the two arbitrary constants $k_1(h)$, $k_2(h)$ appearing in the definition \eqref{T} of the map $T$ can be chosen so that the two constraints \eqref{constraints} are satisfied by $T[h]$, for any $h\in X_{L,\rho}$, provided $L$ is sufficiently large and, in turn, $\rho$ is small enough. The two constants are uniformly close to 1, see \eqref{eq:K1K2bis}; notice that, though a natural guess for the constant $k_1(h)$ is $\frac{1}{1+\rho}$, according to \eqref{eq:system}, the latter is indeed close to 1.

\begin{lemma}\label{lem:K1K2}
	There exist $L_1\geq L_0$ and a map $\rho_1:(L_1,\infty)\to(0,\rho_0)$ such that for every $L> L_1$ and $\rho\in(0,\rho_1(L))$ the following properties hold.
	For every $h\in X_{L,\rho}$ we can choose $k_1(h)$, $k_2(h)$ in \eqref{T} such that
	\begin{equation} \label{eq:K1K2}
	M(T[h])=\inte T[h](x)\de x =1, \qquad M_\gamma(T[h])=\frac{1}{\rho}\inte \expay T[h](y)\de y = \frac{1}{1+\rho}\,.
	\end{equation}
	Moreover
	\begin{equation} \label{eq:K1K2bis}
	|k_i(h)-1| < \frac12, \qquad |k_i(h_1)-k_i(h_2)| < \frac12 \|h_1-h_2\|, \qquad i=1,2.
	\end{equation}
\end{lemma}

\begin{proof}
We have by the definition \eqref{T} of the map $T$
\begin{align*}
	M_\gamma(T[h]) & = \frac{k_1(h)}{\rho}\int_{-\infty}^{L\rho}\expay e^{-\psi_\rho[h](y)}\de y + \frac{e^{-aL}k_2(h)}{(1+\rho)a+\rho} + \frac{1}{\rho}\inte \expay\widetilde{R}[h](y)\de y \\
	& =: (1+\lambda_1)k_1(h) + \frac{e^{-aL}}{(1+\rho)a+\rho}k_2(h) + \nu_1\,,
\end{align*}
\begin{align*}
M(T[h]) &= k_1(h)\int_{-\infty}^{L\rho}e^{-\psi_\rho[h](x)}\de x + k_2(h) + \inte \widetilde{R}[h](x)\de x \\
&=: \lambda_2 k_1(h) + k_2(h) + \nu_2\,.
\end{align*}
We hence have to show that for every $h\in X_{L,\rho}$ we can find a solution $(k_1,k_2)$ to the linear system
\begin{align} \label{eq:system}
\begin{cases}
(1+\lambda_1)k_1 + \frac{e^{-aL}}{(1+\rho)a+\rho}k_2 = \frac{1}{1+\rho} - \nu_1 \\
\lambda_2 k_1 + k_2 = 1-\nu_2
\end{cases}
\end{align}
satisfying in addition the conditions \eqref{eq:K1K2bis}.
This will be achieved by showing that the matrix of the system
\begin{equation*}
A :=
\left(
\begin{matrix}
1+\lambda_1 &  \frac{e^{-aL}}{(1+\rho)a+\rho} \\
\lambda_2 & 1
\end{matrix}
\right)
\end{equation*}
is uniformly close to the identity matrix, and that $\nu_1$, $\nu_2$ are uniformly close to 0, with in addition uniform estimates on the Lipschitz continuity of the coefficients (with respect to $h$), provided $L$ is large enough and, in turn, $\rho$ is small enough (depending on $L$).

We hence start by estimating the coefficient $\lambda_1$. For $y\in(-\infty,L\rho)$, first observe that, by setting
$$
t_1 = -\psi_\rho[h](y) = -\frac{1}{\rho}\int_{y}^{L\rho}\expaz Q[h](z)\de z\,,
\qquad
t_2 = \frac{1}{a}(e^{-aL}-e^{-\frac{a}{\rho}y})=-\frac{1}{\rho}\int_{y}^{L\rho}e^{-\frac{a}{\rho}z}\de z\,,
$$
one has using Lemma~\ref{lem:Q}
$$
\max\{e^{t_1},e^{t_2}\} \leq e^{-\frac{1}{2}\psi_\rho[h](y)}\,.
$$
Then by the elementary estimate $|e^{t_1}-e^{t_2}|\leq \max\{e^{t_1},e^{t_2}\}|t_1-t_2|$ it follows that
\begin{align*}
\Big|e^{-\psi_\rho[h](y)}- e^{\frac{1}{a}(e^{-aL}-e^{-\frac{a}{\rho}y})}\Big|
& \leq e^{-\frac{1}{2}\psi_\rho[h](y)} \bigg|\frac{1}{\rho}\int_{y}^{L\rho}\expaz \bigl(Q[h](z)-1\bigr)\de z\bigg| \\
& \upupref{lem:Q2bis}{lem:Q2ter}{\leq} Ce^{(2a+2)L} e^{-\frac{1}{2}\psi_\rho[h](y)} \int_{y}^{L\rho}\expaz \de z \\
& \leq C \rho e^{(2a+2)L} e^{-\frac{1}{2}\psi_\rho[h](y)} \expay\,.
\end{align*}
By using this estimate we obtain
\begin{align} \label{eq:proofKK1}
|\lambda_1| & = \bigg| \frac{1}{\rho}\int_{-\infty}^{L\rho}\expay e^{-\psi_\rho[h](y)}\de y -1 \bigg| \nonumber \\
& = \bigg| \frac{1}{\rho}\int_{-\infty}^{L\rho}\expay \Bigl(e^{-\psi_\rho[h](y)}- e^{\frac{1}{a}(e^{-aL}-e^{-\frac{a}{\rho}y})}\Bigr)\de y \bigg| \nonumber\\
& \leq Ce^{(2a+2)L} \biggl[\int_{-\infty}^{-L\rho} e^{-\frac{1}{2}\psi_\rho[h](y)} e^{-\frac{2a}{\rho}y}\de y + \int_{-L\rho}^{L\rho} e^{-\frac{1}{2}\psi_\rho[h](y)} e^{-\frac{2a}{\rho}y}\de y\biggr] \nonumber\\
& \xupref{lem:psi2}{\leq} Ce^{(2a+2)L} \biggl[\int_{-\infty}^{-L\rho} e^{\frac{m}{\rho}(y+L\rho)} e^{-\frac{2a}{\rho}y}\de y + \int_{-L\rho}^{L\rho} e^{-\frac{2a}{\rho}y}\de y\biggr]
\leq C\rho e^{(4a+2)L}\,.
\end{align}
We turn to the estimate for $\lambda_2$: we have by Lemma~\ref{lem:psi}
\begin{align} \label{eq:proofKK2}
|\lambda_2| & = \int_{-\infty}^{L\rho} e^{-\psi_\rho[h](x)}\de x
\leq \int_{-\infty}^{-L\rho}e^{-\psi_\rho[h](-L\rho)}e^{\frac{2m}{\rho}(x+L\rho)}\de x + \int_{-L\rho}^{L\rho}\de x
\leq C\rho + 2L\rho\,. 
\end{align}
We can easily obtain bounds on $\nu_1$ and $\nu_2$ by using \eqref{eq:est-rem3-4}:
\begin{align} \label{eq:proofKK3}
|\nu_1| 
&\leq  \frac{1}{\rho}\inte \expay \big| \widetilde{R}[h](y)\big| \de y
\leq C e^{(m+a+2)L} \int_{-\infty}^{-L\rho} e^{\frac{m}{\rho}(y+L\rho)}\expay\de y \nonumber \\
& \qquad + C e^{(a+2)L}\int_{-L\rho}^{L\rho}\expay\de y
+ C\bigl(e^{-aL}+\rho e^{aL}\bigr)\frac{1}{\rho}\int_{L\rho}^\infty e^{-\frac12(y-L\rho)}\expay\de y \nonumber \\
& \leq C\bigl( e^{-aL} + \rho e^{(m+2a+2)L} \bigr)\,.
\end{align}
Similarly,
\begin{align} \label{eq:proofKK4}
|\nu_2| \leq 
\inte \big| \widetilde{R}[h](x)\big| \de x
& \leq C\rho e^{(m+a+2)L}\int_{-\infty}^{-L\rho} e^{\frac{m}{\rho}(x+L\rho)}\de x
+ C\rho e^{(a+2)L}\int_{-L\rho}^{L\rho}\de x \nonumber \\
&\qquad + C\bigl(e^{-aL}+\rho e^{aL}\bigr)\int_{L\rho}^\infty e^{-\frac12(x-L\rho)}\de x \nonumber \\
& \leq C\bigl( e^{-aL} + \rho e^{(m+a+2)L} \bigr)\,.
\end{align}

From the previous estimates it is clear that the quantities $|A-\mathrm{Id}|$, $|\nu_1|$, $|\nu_2|$ can be made arbitrary small, for every $L$ large enough and for every $\rho\in(0,\rho(L))$, where $\rho(L)$ is a sufficiently small value depending on $L$. Hence the linear system \eqref{eq:system} has a unique solution $(k_1,k_2)$ satisfying the first condition in \eqref{eq:K1K2bis}.

To obtain also the second condition in \eqref{eq:K1K2bis}, it is sufficient to show that the Lipschitz constant of the coefficients $\lambda_i$, $\nu_i$ can be made arbitrarily small. We hence write explicitly the dependence of the coefficients on the function $h\in X_{L,\rho}$.
By \eqref{lem:psi3} (with $y=L\rho$) we have
\begin{align*}
|\lambda_1(h_1)-\lambda_1(h_2)| &\leq \frac{1}{\rho}\int_{-\infty}^{L\rho} e^{-\frac{a}{\rho}x}\big| e^{-\psi_\rho[h_1](x)}-e^{-\psi_\rho[h_2](x)}\big|\de x \\
& \leq C e^{(3a+2)L}\|h_1-h_2\| \biggl( \int_{-\infty}^{-L\rho} e^{\frac{2m}{\rho}(x+L\rho)}e^{-\frac{2a}{\rho}x}\de x + \int_{-L\rho}^{L\rho} e^{-\frac{a}{\rho}x}\de x \biggr) \\
& \leq C \rho e^{(5a+2)L}\|h_1-h_2\|\,.
\end{align*}
Similarly,
\begin{align*}
|\lambda_2(h_1)-\lambda_2(h_2)|
\leq \int_{-\infty}^{L\rho} \big| e^{-\psi_\rho[h_1](x)}-e^{-\psi_\rho[h_2](x)}\big|\de x
\leq C \rho e^{(4a+2)L}\|h_1-h_2\|\,.
\end{align*}
To obtain similar estimates for $\nu_1$, $\nu_2$, we argue as in \eqref{eq:proofKK3}--\eqref{eq:proofKK4}, using this time \eqref{eq:est-rem3-4bis}:
\begin{align*}
|\nu_1(h_1)-\nu_1(h_2)| \leq \frac{1}{\rho}\inte\expay \big|\widetilde{R}[h_1](y)-\widetilde{R}[h_2](y)\big|\de y \leq C\bigl(e^{-aL}+\rho e^{(m+5a+4)L}\bigr)\|h_1-h_2\|\,,
\end{align*}
\begin{align*}
|\nu_2(h_1)-\nu_2(h_2)| \leq \inte \big|\widetilde{R}[h_1](y)-\widetilde{R}[h_2](y)\big|\de y \leq C\bigl(e^{-aL}+\rho e^{(m+5a+4)L}\bigr)\|h_1-h_2\|\,.
\end{align*}
It follows from the previous estimates that the Lipschitz constants of the coefficients $\lambda_i$, $\nu_i$ can be made arbitrarily small by possibly taking a larger $L$ and, in turn, a smaller $\rho(L)$. The conclusion of the lemma follows.
\end{proof}


\section{Asymptotic decay of the solution at \texorpdfstring{$-\infty$}{- infinity}} \label{sect:decay-}

The fixed point $h$ obtained in Theorem~\ref{thm:fixedpoint} is by construction a solution to \eqref{equation} satisfying the constraints \eqref{constraints} and the decay estimate
$$
|h(x)| \leq 2e^{\omega(x)}\,,
$$
where the exponent $\omega$ is defined in \eqref{parameters3}.
In this section we refine the previous bound in the region $\{x< L\rho\}$, and in particular we determine the exact asymptotics of the solution as $x\to-\infty$, by showing that
\begin{equation} \label{eq:asimp-}
h(x) \sim e^{-\psi_\rho(x)} \qquad\text{as }x\to-\infty,
\end{equation}
where $\psi_\rho$ is defined in \eqref{psi}.
Since $h$ is now a fixed function, to simplify the notation we omit the dependence on $h$ in the functions $Q$ and $\psi_\rho$ defined in \eqref{Q} and \eqref{psi} respectively. As before, $C$ will always denote a generic constant depending only on the kernel $K$ and on $m$, but not on $L$ and $\rho$, which may change from line to line. The main result of this section is the following.

\begin{theorem}[Decay at $-\infty$] \label{thm:decay-}
	Let $h\in X_{L,\rho}$ be the solution to \eqref{equation}--\eqref{constraints} determined in Theorem~\ref{thm:fixedpoint}, for $L>\bar{L}$ and $\rho\in(0,\bar{\rho}(L))$. Then, by possibly choosing a smaller $\bar{\rho}(L)$,
	\begin{equation} \label{eq:decay-}
	\textstyle\frac14 e^{-\psi_\rho(x)} \leq h(x) \leq 2e^{-\psi_\rho(x)} \qquad\text{for all }x<L\rho\,.
	\end{equation}
\end{theorem}

\begin{remark}\label{rm:decay-}
	Observe that the bounds \eqref{eq:decay-} involve the function $\psi_\rho$, depending on the solution $h$ itself. However, we can also characterize the decay of $h$ at $-\infty$ by explicit functions: recalling that, by Lemma~\ref{lem:Q}, one has $|Q[h]-1|\leq\frac12$, it easily follows from the definition of $\psi_\rho$ that
	\begin{equation*}
	\frac14 e^{\frac{2}{a}(e^{-aL}-e^{-\frac{a}{\rho}x})} \leq h(x) \leq 2e^{\frac{1}{2a}(e^{-aL}-e^{-\frac{a}{\rho}x})} \qquad\text{for all } x<L\rho\,.
	\end{equation*}
	In fact, as $Q[h]$ converges uniformly to the constant 1 as $\rho\to0$ by \eqref{lem:Q2bis}--\eqref{lem:Q2ter}, the decay of $h$ at $-\infty$ gets arbitrarily close to that of the function $e^{\frac{1}{a}(e^{-aL}-e^{-\frac{a}{\rho}x})}$.
\end{remark}

\begin{proof}[Proof of Theorem~\ref{thm:decay-}]
The construction in Section~\ref{sect:fixedpoint} show that $h$ obeys
\begin{equation}\label{eq:proof-1}
h(x) = k_1 e^{-\psi_\rho(x)} + \widetilde{R}[h](x) \qquad\text{for }x<L\rho,
\end{equation}
where the remainder $\widetilde{R}$ is defined in \eqref{eq:remainder3}, and $k_1\in(\frac12,\frac32)$ by Lemma~\ref{lem:K1K2}.
By inserting the bound \eqref{eq:est-rem3-4} in \eqref{eq:proof-1}, we have
\begin{align*}
|h(x)| \leq k_1 e^{-\psi_\rho(x)} +
\begin{cases}
C\rho e^{(m+a+2)L}e^{\frac{m}{\rho}(x+L\rho)} & \text{for }x\leq -L\rho, \\
C\rho e^{(a+2)L} & \text{for }-L\rho < x < L\rho.
\end{cases}
\end{align*}
Hence, by possibly choosing a smaller $\bar{\rho}(L)$, we obtain the bound
\begin{equation*}
\textstyle |h(x)| \leq 2 e^{-\psi_\rho(x)} + \sigma e^{\omega(x)} \qquad\text{for }x<L\rho,
\end{equation*}
for some universal constant $\sigma\in(0,1)$, where $\omega$ is defined in \eqref{parameters3}.
An application of Lemma~\ref{lem:decay-} below yields that $h$ satisfies \eqref{eq:claim-}.
In turn, by reducing one more time the constant $\bar{\rho}(L)$ so that $\overline{C}(L)\bar{\rho}(L)<\frac14$, we obtain by combining \eqref{eq:proof-1} and \eqref{eq:claim-}
\begin{equation*}
\textstyle |h(x)| \leq 2e^{-\psi_\rho(x)} + \sigma^2 e^{\omega(x)} \qquad\text{for }x<L\rho.
\end{equation*}

We can then iterate the previous argument, applying Lemma~\ref{lem:decay-} with $\sigma$ replaced by $\sigma^2$: after $n$ step we eventually end up with the inequality
\begin{equation*}
|\widetilde{R}[h](x)| \leq \frac14 e^{-\psi_\rho(x)} + \sigma^{2^n} e^{\omega(x)} \qquad\text{for }x<L\rho.
\end{equation*}
The conclusion now follows by using this estimate in \eqref{eq:proof-1} and letting $n\to\infty$.
\end{proof}

\begin{lemma}\label{lem:decay-}
Assume that $h$ satisfies the inequality
\begin{equation} \label{eq:assumpt-}
\textstyle |h(x)| \leq 2e^{-\psi_\rho(x)} + \sigma e^{\omega(x)} \qquad\text{for }x<L\rho,
\end{equation}
for a constant $\sigma\in(0,1)$. Then
\begin{equation} \label{eq:claim-}
|\widetilde{R}[h](x)| \leq \overline{C}(L)\rho \Bigl( e^{-\psi_\rho(x)} + \sigma^2 e^{\omega(x)} \Bigr) \qquad\text{for }x<L\rho\,,
\end{equation}
for a constant $\overline{C}(L)$ depending on $L$, but not on $\rho$ and $\sigma$.
\end{lemma}

\begin{proof}
We prove the lemma by using the bound \eqref{eq:assumpt-} in the formulas \eqref{eq:remainder1} and \eqref{eq:remainder3} for the remainder terms $R[h]$ and $\widetilde{R}[h]$. We assume in the following that $x<L\rho$. As
\begin{align} \label{eq:R0}
|R[h](x)| \leq \frac{1}{\rho}\int_{-\infty}^x \de y \int_{-\infty}^x \expay e^{\frac{a}{\rho}(y-z)} K(\expyz,1)|h(y)||h(z)| \de z\,,
\end{align}
it is convenient to first find a bound for the inner integral
\begin{align*}
P(x,y) := \int_{-\infty}^x e^{\frac{a}{\rho}(y-z)} K(\expyz,1)|h(z)|\de z\,,
\qquad\qquad\text{for }y\leq x <L\rho.
\end{align*}

\medskip
\noindent\textit{Step 1: bounds on $P(x,y)$.}
We claim that
\begin{equation}\label{eq:P0}
P(x,y) \leq C(L)\rho e^{\frac{a}{\rho}x}e^{-\psi_\rho(x)} + C(L)\sigma\rho e^{\omega(x)}\,,
\qquad\text{for }x < L\rho,\, y\leq x,
\end{equation}
for a constant $C(L)$ depending on $L$ but not on $h$, $\rho$, and $\sigma$.
By \eqref{kernel3}--\eqref{kernel4} we have
\begin{align} \label{eq:P1}
P(x,y)
& \leq C \int_{-\infty}^y e^{\frac{a+1}{\rho}(y-z)}|h(z)|\de z + C\int_{y}^x |h(z)|\de z \nonumber \\
& \xupref{eq:assumpt-}{\leq} C \int_{-\infty}^y e^{\frac{a+1}{\rho}(y-z)} \Bigl(e^{-\psi_\rho(z)} + \sigma e^{\omega(z)} \Bigl)\de z
+ C \int_{y}^x\Bigl(e^{-\psi_\rho(z)} + \sigma e^{\omega(z)} \Bigl)\de z \nonumber \\
& \leq C\int_{-\infty}^y e^{\frac{a+1}{\rho}(y-z)}e^{-\psi_\rho(z)}\de z + C\int_y^x e^{-\psi_\rho(z)}\de z + C(L)\sigma\rho e^{\omega(x)}\,,
\end{align}
where the bound on the integrals involving the term $e^{\omega(z)}$ is obtained by elementary computations, using the explicit form of $\omega(z)$.
We next consider the first integral in \eqref{eq:P1}: with the notation $\bar{y}:=\min\{y,-L\rho\}$ we have, using in particular the monotonicity property $e^{-\frac12\psi_\rho(z)-\frac{z}{\rho}}\leq e^{-\frac12\psi_\rho(\bar{y})-\frac{\bar{y}}{\rho}}$ for $z<\bar{y}$, which is proved in Lemma~\ref{lem:psi},
\begin{align} \label{eq:P2}
\int_{-\infty}^y e^{\frac{a+1}{\rho}(y-z)}e^{-\psi_\rho(z)}\de z
& \leq \int_{-\infty}^{\bar{y}} e^{\frac{a+1}{\rho}(y-z)}e^{-\psi_\rho(z)}\de z
+ e^{-\psi_\rho(x)} \int_{\bar{y}}^y e^{\frac{a+1}{\rho}(y-z)}\de z \nonumber\\
& \leq e^{\frac{a}{\rho}y} e^{\frac{1}{\rho}(y-\bar{y})} e^{-\frac12\psi_\rho(\bar{y})} \int_{-\infty}^y e^{-\frac{a}{\rho}z}e^{-\frac12\psi_\rho(z)}\de z
+ C\rho e^{(a+2)L}e^{\frac{a}{\rho}x}e^{-\psi_\rho(x)} \nonumber\\
& \xupref{lem:Q2}{\leq} 2e^{2L}e^{\frac{a}{\rho}y}e^{-\frac12\psi_\rho(y)} \int_{-\infty}^y e^{-\frac{a}{\rho}z}Q(z)e^{-\frac12\psi_\rho(z)}\de z
+ C(L)\rho e^{\frac{a}{\rho}x}e^{-\psi_\rho(x)} \nonumber\\
& \xupref{eq:derivativepsi}{=} 4 e^{2L} \rho e^{\frac{a}{\rho}y}e^{-\psi_\rho(y)}
+ C(L)\rho e^{\frac{a}{\rho}x}e^{-\psi_\rho(x)}\,.
\end{align}
For the second integral in \eqref{eq:P1} we have instead
\begin{align} \label{eq:P3}
\int_y^x e^{-\psi_\rho(z)}\de z
\xupref{lem:Q2}{\leq} 2 e^{\frac{a}{\rho}x}\int_y^x e^{-\frac{a}{\rho}z}Q(z)e^{-\psi_\rho(z)}\de z
\leq 2\rho e^{\frac{a}{\rho}x} e^{-\psi_\rho(x)}\,,
\end{align}
where we again computed explicitly the integral, recalling \eqref{eq:derivativepsi}.
The claim \eqref{eq:P0} follows now from \eqref{eq:P1}, \eqref{eq:P2}, and \eqref{eq:P3}.

\medskip
\noindent\textit{Step 2: bounds on $R[h]$.}
We now claim that
\begin{equation} \label{eq:R1}
|R[h](x)|\leq C(L)\rho \bigl( e^{-2\psi_\rho(x)} + \sigma^2 e^{\omega(x)} \bigr)
\qquad\text{for }x < L\rho
\end{equation}
(for a possibly larger constant $C(L)$, depending as before on $L$ but not on $h$, $\rho$, and $\sigma$).
The claim is obtained by inserting \eqref{eq:P0} into \eqref{eq:R0}. Indeed,
\begin{align} \label{eq:R2}
|R[h](x)|
& \leq \frac{1}{\rho}\int_{-\infty}^x \expay P(x,y) |h(y)| \de y \nonumber\\
& \upupref{eq:assumpt-}{eq:P0}{\leq} C(L) \Bigl( e^{\frac{a}{\rho}x}e^{-\psi_\rho(x)} + \sigma e^{\omega(x)} \Bigr) \biggl( \int_{-\infty}^x \expay e^{-\psi_\rho(y)}\de y + \sigma\int_{-\infty}^x \expay e^{\omega(y)}\de y\biggr) \nonumber\\
& \leq C(L) \Bigl( e^{\frac{a}{\rho}x}e^{-\psi_\rho(x)} + \sigma e^{\omega(x)} \Bigr) \Bigl( \rho e^{-\psi_\rho(x)} + \sigma\rho e^{\omega(x)}e^{-\frac{a}{\rho}x} \Bigr) \nonumber \\
& \leq C(L)\rho \Bigl( e^{-2\psi_\rho(x)}e^{\frac{a}{\rho}x} + \sigma^2 e^{2\omega(x)}e^{-\frac{a}{\rho}x} \Bigr)\,,
\end{align}
where to go from the second to the third line (where we possibly have a larger constant $C(L)$) we computed explicitly the two integrals, and in particular
\begin{equation*}
\int_{-\infty}^{x}e^{-\frac{a}{\rho}y}e^{-\psi_\rho(y)}\de y
\xupref{lem:Q2}{\leq} 2\int_{-\infty}^x e^{-\frac{a}{\rho}y}Q(y)e^{-\psi_\rho(y)}\de y
\xupref{eq:derivativepsi}{=} 2\rho e^{-\psi_\rho(x)}\,.
\end{equation*}
The last inequality in \eqref{eq:R2} follows instead by Young's inequality.
Hence \eqref{eq:R1} follows by observing that $e^{\frac{a}{\rho}x}\leq e^{aL}$, $e^{\omega(x)}e^{-\frac{a}{\rho}x}\leq e^{aL}$.

\medskip
\noindent\textit{Step 3: bounds on $\widetilde{R}[h]$.}
We finally turn to the proof of \eqref{eq:claim-}. By the definition \eqref{eq:remainder3} of $\widetilde{R}[h]$ and by the previous step we have
\begin{align*}
|\widetilde{R}[h](x)|
& \leq |R[h](x)|+ \frac{1}{\rho}\int_x^{L\rho} e^{-(\psi_\rho(x)-\psi_\rho(y))} \expay Q(y) |R[h](y)| \de y \\
& \xupref{eq:R1}{\leq} C(L)\rho \bigl( e^{-2\psi_\rho(x)} + \sigma^2 e^{\omega(x)} \bigr) + C(L) e^{-\psi_\rho(x)} \int_{x}^{L\rho} e^{-\psi_\rho(y)} \expay Q(y)\de y \\
& \qquad\qquad + C(L)\sigma^2\int_x^{L\rho} e^{-(\psi_\rho(x)-\psi_\rho(y))} e^{-\frac{a}{\rho}y} Q(y) e^{\omega(y)}\de y \\
& \xupref{eq:derivativepsi}{\leq} C(L)\rho \bigl( e^{-2\psi_\rho(x)} + \sigma^2 e^{\omega(x)} \bigr) + C(L)\rho e^{-\psi_\rho(x)} \\
& \qquad\qquad + C(L)\sigma^2 e^{-\psi_\rho(x)} \int_x^{L\rho} e^{\psi_\rho(y)} e^{-\frac{a}{\rho}y} Q(y) e^{\omega(y)}\de y\,.
\end{align*}
Hence to prove \eqref{eq:claim-} it only remains to estimate the last term in the previous expression by a quantity of the form $C(L)\rho ( e^{-\psi_\rho(x)} + \sigma^2 e^{\omega(x)})$. For $x\in(-L\rho,L\rho)$ we have
\begin{align*}
\sigma^2 e^{-\psi_\rho(x)} \int_x^{L\rho} e^{\psi_\rho(y)} e^{-\frac{a}{\rho}y} Q(y) e^{\omega(y)} \de y
& \leq \sigma^2e^{-\psi_\rho(x)} \int_{x}^{L\rho} e^{\psi_\rho(y)} e^{-\frac{a}{\rho}y} Q(y) \de y \\
& \xupref{eq:integralpsi}{\leq} \sigma^2 \rho = \sigma^2\rho e^{\omega(x)}\,.
\end{align*}
In the case $x\leq -L\rho$ we have instead
\begin{align*}
\sigma^2 e^{-\psi_\rho(x)} & \int_x^{L\rho} e^{\psi_\rho(y)} e^{-\frac{a}{\rho}y}Q(y) e^{\omega(y)}\de y \\
& \leq e^{-\psi_\rho(x)}\int_{-L\rho}^{L\rho} e^{\psi_\rho(y)} e^{-\frac{a}{\rho}y} Q(y) \de y
+ \sigma^2 e^{-\psi_\rho(x)} \int_x^{-L\rho} e^{\psi_\rho(y)} e^{-\frac{a}{\rho}y}Q(y)e^{\frac{m}{\rho}(y+L\rho)} \de y \\
& \upupref{eq:integralpsi}{lem:psi2}{\leq} \rho e^{-\psi_\rho(x)} e^{\psi_\rho(-L\rho)} + \sigma^2 e^{-\frac12\psi_\rho(x)} e^{\frac{m}{\rho}(x+L\rho)} \int_x^{-L\rho} e^{\frac12\psi_\rho(y)}e^{-\frac{a}{\rho}y}Q(y)\de y \\
& \leq C(L)\rho e^{-\psi_\rho(x)} + 2\sigma^2 \rho e^{\omega(x)}\,,
\end{align*}
where in the last passage we used in particular the bound
\begin{equation*}
\psi_\rho(-L\rho) = \frac{1}{\rho}\int_{-L\rho}^{L\rho} e^{-\frac{a}{\rho}y}Q(y)\de y
\xupref{lem:Q2}{\leq} \frac{3}{2\rho}\int_{-L\rho}^{L\rho}e^{-\frac{a}{\rho}y}\de y
= \frac{3}{2a}\bigl( e^{aL}-e^{-aL} \bigr) \,.
\end{equation*}
This completes the proof of the lemma.
\end{proof}


\section{Asymptotic decay of the solution at \texorpdfstring{$\infty$}{infinity}} \label{sect:decay+}

The last tile missing in the proof of Theorem~\ref{thm:exist} is the behaviour of the solution in the region $\{x>L\rho\}$, and in particular the non-negativity of $h$ in this interval and the characterization of its exact decay as $x\to\infty$.
Let therefore $h\in X_{L,\rho}$ be the solution to \eqref{equation}--\eqref{constraints} constructed in Theorem~\ref{thm:fixedpoint}, for $L>\bar{L}$ and $\rho\in(0,\bar{\rho}(L))$. 
In the main result of this section, Theorem~\ref{thm:decay+}, we show that $h$ satisfies
\begin{equation} \label{eq:asimp+}
h(x) \sim e^{-x} \qquad\text{as }x\to\infty.
\end{equation}
The proof requires a preliminary bound on the Lipschitz constant of the solution, which we establish in the following lemma.
We recall that by Theorem~\ref{thm:decay-} we have for every $x\in\R$
\begin{equation} \label{parameters4}
|h(x)| \leq 2e^{\lambda(x)}, \qquad\qquad
 \lambda(x):=
 \begin{cases}
  -\psi_\rho(x) & \text{for }x< L\rho,\\
 -\frac12(x-L\rho) & \text{for }x\geq L\rho.
 \end{cases}
\end{equation}
The bound for $x<L\rho$ is given implicitly in terms of the function $\psi_\rho$ introduced in \eqref{psi}, but as a consequence of \eqref{lem:psi2} one easily finds
\begin{equation} \label{parameters4b}
|h(x)| \leq 2e^{-\psi_\rho(x)} \leq 2e^{\frac{2m}{\rho}(x+L\rho)} \qquad\text{for }x<-L\rho,
\end{equation}	
which guarantees in particular an exponential decay of $h$ as $x\to-\infty$.
We introduce a notation for the difference quotient (for $\tau>0$)
\begin{equation} \label{diffquot}
D_\tau h(x) := \frac{h(x) - h(x-\tau)}{\tau}\,.
\end{equation}

\begin{lemma} \label{lem:diffquot}
There exists a constant $A_L$ (depending on the kernel and on $L$, but not on $\tau$ and $\rho$), such that, by possibly choosing a smaller $\bar{\rho}(L)$,
\begin{equation} \label{eq:diffquot}
|D_\tau h(x)| \leq \frac{A_L}{\rho}e^{-\frac12(x-L\rho)} \qquad\text{for every }x\geq L\rho,
\end{equation}
for all $\tau\in(0,\rho)$ and $\rho\in(0,\bar{\rho}(L))$.
\end{lemma}

\begin{proof}
We set
\begin{equation*}
\Delta_\tau(x) := \sup_{y\in(x-\rho,x)}|D_\tau h(y)|\,.
\end{equation*}
The conclusion will be achieved by proving a uniform decay estimate on $\Delta_\tau$, using an iteration argument.
Since we do not need to keep track of the dependence on $L$ of the constants appearing in the following estimates, along this proof the letter $c$ will denote a generic constant, depending only on the kernel and $L$, which may change from line to line. We divide the proof into two steps.

\medskip\noindent
\textit{Step 1.} We claim that for all $x\in\R$ and $\tau\in(0,\rho)$
\begin{equation} \label{diffquot0}
|D_\tau h(x)| \leq c\Bigl( 1 + \frac{1}{\rho}e^{-\frac{a}{\rho}x} \Bigr) e^{\lambda(x)}
+ c\Delta_\tau(x)\int_{-\infty}^x e^{-\frac{a}{\rho}y}|h(y)|\de y
+ c e^{\lambda(x)} \int_{-\infty}^x e^{-\frac{a}{\rho}y} |D_\tau h(y)|\de y \,.
\end{equation}
To prove \eqref{diffquot0}, we apply the difference quotient operator $D_\tau$ to the equation \eqref{equation}: in particular, we split the integral on the right-hand side of \eqref{equation} in the two regions above and below the $y$-axis (see Figure~\ref{fig:domains}), and we obtain after a change of variables in the second integral
\begin{align} \label{diffquot1}
|D_\tau h(x)|
& \leq \frac{1}{\tau\rho}\bigg| \int_{x-\tau}^x \int_x^\infty \expaz K(e^\frac{y-z}{\rho},1)h(y)h(z)\de z \de y\bigg| \nonumber\\
& \qquad + \frac{1}{\tau\rho}\bigg|\int_{-\infty}^{x-\tau} \int_{x-\tau}^x \expaz K(e^\frac{y-z}{\rho},1)h(y)h(z)\de z \de y \bigg| \nonumber\\
& \qquad + \frac{1}{\tau\rho} \iint_{\Omega_\rho} K(e^\frac{y-z}{\rho},1) \Big| \expaz h(y)h(z) - e^{-\frac{a}{\rho}(z-\tau)}h(y-\tau)h(z-\tau)\Big| \de y\de z \,,
\end{align}
where $\Omega_\rho$ is defined in \eqref{eq:regions}.
For the first term in \eqref{diffquot1} we have by \eqref{kernel3}
\begin{align} \label{diffquot2}
\frac{1}{\tau\rho}\bigg|\int_{x-\tau}^x \int_x^\infty \expaz K(e^\frac{y-z}{\rho},1)h(y)h(z)\de z \de y\bigg|
& \leq \frac{c}{\tau\rho}\int_{x}^\infty|h(z)|\de z \int_{x-\tau}^x \expay|h(y)|\de y \nonumber \\
& \leq \frac{c}{\rho} e^{-\frac{a}{\rho}x}e^{\lambda(x)}\,;
\end{align}
in the last inequality we used the bound $\int_x^\infty|h(x)|\de x \leq \int_{-\infty}^\infty |h(x)|\de x\leq c$ (which follows directly from the decay $|h(x)|\leq 2e^{\omega(x)}$ in the definition \eqref{space} of the space $X_{L,\rho}$), together with $|h(y)|\leq 2e^{\lambda(y)}\leq ce^{\lambda(x)}$ for $y\in(x-\tau,x)$. Similarly, the second term in \eqref{diffquot1} is bounded by
\begin{align} \label{diffquot3}
\frac{1}{\tau\rho}\bigg|\int_{-\infty}^{x-\tau} \int_{x-\tau}^x \expaz K(e^\frac{y-z}{\rho},1)h(y)h(z)\de z \de y \bigg|
& \leq \frac{c}{\tau\rho}\int_{-\infty}^{x-\tau}\expay|h(y)|\de y \int_{x-\tau}^x |h(z)|\de z \nonumber \\
& \leq c e^{\lambda(x)}\,,
\end{align}
where we used also \eqref{eq:prelim1} in the last inequality.
It remains to estimate the third term in \eqref{diffquot1}, which can be written as
\begin{align} \label{diffquot4}
\frac{1}{\tau\rho} \iint_{\Omega_\rho} &K(e^\frac{y-z}{\rho},1) \Big| \expaz h(y)h(z) - e^{-\frac{a}{\rho}(z-\tau)}h(y-\tau)h(z-\tau)\Big| \de y\de z \nonumber\\
& \leq \frac{1}{\rho}\iint_{\Omega_\rho} K(e^\frac{y-z}{\rho},1) \expaz |D_\tau h(y)||h(z)|\de y \de z \nonumber\\
& \qquad + \frac{1}{\rho}\iint_{\Omega_\rho} K(e^\frac{y-z}{\rho},1) \expaz |h(y-\tau)||D_\tau h(z)|\de y \de z \nonumber\\
& \qquad + \frac{1}{\tau\rho}\iint_{\Omega_\rho} K(e^\frac{y-z}{\rho},1) |e^{-\frac{a}{\rho}z}-e^{-\frac{a}{\rho}(z-\tau)}| |h(y-\tau)||h(z-\tau)|\de y \de z \nonumber\\
& =: J_1 + J_2 + J_3\,.
\end{align}
We split each of the three integrals $J_i$ into two terms $J_{i,1}$ and $J_{i,2}$, defined as the corresponding integrals over the two domains $A_\rho$ and $B_\rho$, introduced in \eqref{eq:regions}, respectively. The behaviour of the kernel in these regions is given by \eqref{kernelA} and \eqref{kernelB}.
Then we have for $J_1$ the bounds
\begin{align} \label{diffquot5}
J_{1,1}
&= \frac{1}{\rho}\iint_{A_\rho} K(e^\frac{y-z}{\rho},1)\expaz|D_\tau h(y)||h(z)|\de y \de z \nonumber\\
& \xupref{kernelA}{\leq} \frac{c}{\rho}\int_{x-\rho\ln2}^{x} \de z\, |h(z)| \int_{x+\rho\ln(1-e^\frac{z-x}{\rho})}^x \expay |D_\tau h(y)|\de y \nonumber\\
&\leq c e^{\lambda(x)} \int_{-\infty}^x \expay |D_\tau h(y)|\de y\,,
\end{align}
and
\begin{align} \label{diffquot6}
J_{1,2}
&= \frac{1}{\rho}\iint_{B_\rho} K(e^\frac{y-z}{\rho},1)\expaz|D_\tau h(y)||h(z)|\de y \de z \nonumber\\
& \xupref{kernelB}{\leq} \frac{c}{\rho} \int_{-\infty}^{x-\rho\ln 2} \de z \int_{x+\rho\ln(1-e^\frac{z-x}{\rho})}^x e^{\frac{y}{\rho}} e^{-\frac{a+1}{\rho}z}|D_\tau h(y)| |h(z)| \de y \nonumber\\
& \leq \frac{c}{\rho} \Delta_\tau(x) \int_{-\infty}^{x-\rho\ln2} \de z \,e^{-\frac{a+1}{\rho}z} |h(z)|\int_{x+\rho\ln(1-e^\frac{z-x}{\rho})}^x e^\frac{y}{\rho}\de y \nonumber\\
& = c\Delta_\tau(x) \int_{-\infty}^{x-\rho\ln2} \expaz |h(z)|\de z\,.
\end{align}
We next consider the two terms which constitute $J_2$: for the part over $A_\rho$ we have
\begin{align} \label{diffquot7}
J_{2,1}
& \xupref{kernelA}{\leq} \frac{c}{\rho} \int_{x-\rho\ln2}^{x} \de z \, |D_\tau h(z)| \int_{x+\rho\ln(1-e^\frac{z-x}{\rho})}^x \expay |h(y-\tau)|\de y \nonumber\\
& \leq c\Delta_\tau(x) \int_{-\infty}^x \expay |h(y-\tau)|\de y
\leq c\Delta_\tau(x) \int_{-\infty}^x \expay |h(y)|\de y\,,
\end{align}
while for the part over $B_\rho$
\begin{align} \label{diffquot8}
J_{2,2}
& \xupref{kernelB}{\leq} \frac{c}{\rho} \int_{-\infty}^{x-\rho\ln 2} \de z \int_{x+\rho\ln(1-e^\frac{z-x}{\rho})}^x e^{\frac{y}{\rho}} e^{-\frac{a+1}{\rho}z}|D_\tau h(z)| |h(y-\tau)|\de y \nonumber\\
& \leq \frac{c}{\rho} e^{\lambda(x)} \int_{-\infty}^{x-\rho\ln2}\de z \, e^{-\frac{a+1}{\rho}z} |D_\tau h(z)|\int_{x+\rho\ln(1-e^\frac{z-x}{\rho})}^x e^\frac{y}{\rho}\de y \nonumber\\
& = c e^{\lambda(x)} \int_{-\infty}^{x-\rho\ln2} \expaz |D_\tau h(z)|\de z\,.
\end{align}
We finally consider the third integral $J_3$ appearing in \eqref{diffquot4}, which we split as usual in the part over $A_\rho$
\begin{align} \label{diffquot9}
J_{3,1}
& \xupref{kernelA}{\leq} \frac{c|e^{\frac{a\tau}{\rho}}-1|}{\rho\tau} \int_{x-\rho\ln2}^{x} \de z \int_{x+\rho\ln(1-e^\frac{z-x}{\rho})}^x \expay |h(y-\tau)| |h(z-\tau)|\de y \nonumber\\
& \leq \frac{c}{\rho} e^{\lambda(x)} \int_{-\infty}^x \expay |h(y-\tau)|\de y
\xupref{eq:prelim1}{\leq} c e^{\lambda(x)}\,,
\end{align}
and the part over $B_\rho$
\begin{align} \label{diffquot10}
J_{3,2}
& \xupref{kernelB}{\leq} \frac{c|e^{\frac{a\tau}{\rho}}-1|}{\rho\tau} \int_{-\infty}^{x-\rho\ln 2} \de z \int_{x+\rho\ln(1-e^\frac{z-x}{\rho})}^x e^{\frac{y-z}{\rho}} \expaz |h(y-\tau)| |h(z-\tau)|\de y \nonumber\\
& \leq \frac{c}{\rho^2} e^{\lambda(x)} \int_{-\infty}^{x-\rho\ln2}\de z \, e^{-\frac{a+1}{\rho}z} |h(z-\tau)|\int_{x+\rho\ln(1-e^\frac{z-x}{\rho})}^x e^\frac{y}{\rho}\de y \nonumber\\
& = \frac{c}{\rho} e^{\lambda(x)} \int_{-\infty}^{x-\rho\ln2} \expaz |h(z-\tau)|\de z
\xupref{eq:prelim1}{\leq} c e^{\lambda(x)}\,.
\end{align}
Collecting \eqref{diffquot2}--\eqref{diffquot10} and inserting them in \eqref{diffquot1}, we conclude that the claim \eqref{diffquot0} holds.

\medskip\noindent
\textit{Step 2.} We now use \eqref{diffquot0} to get the conclusion by an iteration argument.
We set $A_\ell:=\Delta_\tau(\ell\rho)$ for $\ell\in\Z$, and we prove by induction that
\begin{equation} \label{diffquot11}
A_\ell \leq \frac{\cbar}{\rho} \, e^{-a\ell}e^{\lambda(\ell\rho)}
\qquad\text{for all }\ell\in\Z,\, \ell\leq0,
\end{equation}
and
\begin{equation} \label{diffquot12}
A_\ell \leq \frac{\cbar}{\rho} \, e^{\lambda(\ell\rho)}
\qquad\text{for all }\ell\in\Z,\, \ell\geq0,
\end{equation}
for some uniform constant $\cbar$ (independent of $\rho$, $\tau$ and $\ell$),
for every $\tau\in(0,\rho)$ and for every $\rho\in(0,\bar{\rho}(L))$, with $\bar{\rho}(L)$ small enough.
The conclusion of the lemma follows directly from the previous claims. Indeed, for every $x\geq L\rho$, choosing $\ell\in\Z$ such that $(\ell-1)\rho< x\leq \ell\rho$ we have
\begin{equation*}
|D_\tau h(x)|
\leq \Delta_\tau(\ell\rho)
= A_\ell
\leq \frac{\cbar}{\rho}e^{\lambda(\ell\rho)}
= \frac{\cbar}{\rho}e^{-\frac12(\ell\rho-L\rho)} 
\leq \frac{\cbar}{\rho} e^{-\frac12(x-L\rho)}\,.
\end{equation*}

\medskip
We are therefore left with the proofs of \eqref{diffquot11} and \eqref{diffquot12}.
First notice that by \eqref{parameters4}
\begin{equation} \label{diffquot-1}
\Delta_\tau(x) \leq \frac{c}{\tau}e^{\lambda(x)}\,,
\end{equation}
which implies in particular $\lim_{|x|\to\infty}\Delta_\tau(x)=0$.
By writing the term
\begin{equation*}
\int_{-\infty}^x \expay |D_\tau h(y)|\de y
= \sum_{n=0}^\infty\int_{x-(n+1)\rho}^{x-n\rho} \expay |D_\tau h(y)|\de y
\leq c\rho e^{-\frac{a}{\rho}x} \sum_{n=0}^\infty \Delta_\tau(x-n\rho)e^{an}
\end{equation*}
and inserting this inequality in \eqref{diffquot0} we get
\begin{equation*}
|D_\tau h(x)|
 \leq c\Bigl( 1 + \frac{1}{\rho}e^{-\frac{a}{\rho}x} \Bigr) e^{\lambda(x)}
 + c\Delta_\tau(x)\int_{-\infty}^x e^{-\frac{a}{\rho}y}|h(y)|\de y
 + c\rho e^{-\frac{a}{\rho}x} e^{\lambda(x)} \sum_{n=0}^\infty \Delta_\tau(x-n\rho)e^{an}\,.
\end{equation*}
In turn, using the fact that $\sup_{y\in(x-\rho,x)}\Delta_\tau(y)\leq\Delta_\tau(x)+\Delta_\tau(x-\rho)$, we deduce from \eqref{diffquot0}
\begin{align*}
\Delta_\tau(x)
& \leq c\Bigl( 1 + \frac{1}{\rho}e^{-\frac{a}{\rho}x} \Bigr) e^{\lambda(x)}
+ c \bigl( \Delta_\tau(x) + \Delta_\tau(x-\rho) \bigr) \int_{-\infty}^x \expay |h(y)|\de y \\
& \qquad + c\rho e^{-\frac{a}{\rho}x}e^{\lambda(x)} \sum_{n=0}^\infty \Delta_\tau(x-n\rho)e^{an}\,,
\end{align*}
and by choosing $\bar{\rho}(L)$ sufficiently small we can absorb the terms with $\Delta_\tau(x)$ in the left-hand side and obtain that for every $x\in\R$
\begin{equation*}
\Delta_\tau(x)
\leq c\Bigl( 1 + \frac{1}{\rho}e^{-\frac{a}{\rho}x} \Bigr) e^{\lambda(x)}
+ c \Delta_\tau(x-\rho)\int_{-\infty}^x \expay |h(y)|\de y
+ c\rho e^{-\frac{a}{\rho}x} e^{\lambda(x)}\sum_{n=1}^\infty \Delta_\tau(x-n\rho)e^{an}\,.
\end{equation*}
By computing the previous expression at $x=\ell\rho$, for $\ell\in\Z$, we find that for some uniform constant $\cbar$ (independent of $\rho$, $\tau$ and $\ell$)
\begin{equation} \label{diffquot13}
A_\ell \leq  \frac{\cbar}{4} \Bigl( 1 + \frac{1}{\rho}e^{-a\ell} \Bigr) e^{\lambda(\ell\rho)}
+ \cbar A_{\ell-1} \int_{-\infty}^{\ell\rho}\expay|h(y)|\de y
+ \cbar\rho e^{\lambda(\ell\rho)}\sum_{k=-\infty}^{\ell-1}A_k e^{-ka}\,.
\end{equation}
In particular, for $\ell\leq0$ we have
\begin{align*}
\int_{-\infty}^{\ell\rho}\expay|h(y)|\de y
&\xupref{parameters4}{\leq} 2\int_{-\infty}^{\ell\rho}\expay e^{-\psi_\rho(y)}\de y \\
&\xupref{lem:Q2}{\leq} 4\int_{-\infty}^{\ell\rho}\expay Q(y)e^{-\psi_\rho(y)}\de y
\xupref{eq:derivativepsi}{=} 4\rho e^{-\psi_\rho(\ell\rho)}\,,
\end{align*}
and inserting this estimate in \eqref{diffquot13} we find
\begin{equation} \label{diffquot14}
A_\ell \leq \frac{\cbar}{4} \Bigl( 1 + \frac{1}{\rho}e^{-a\ell} \Bigr) e^{\lambda(\ell\rho)}
+ 4\cbar\rho A_{\ell-1}e^{\lambda(\ell\rho)}
+ \cbar\rho e^{\lambda(\ell\rho)}\sum_{k=-\infty}^{\ell-1}A_k e^{-ka} \qquad\text{for }\ell\leq0.
\end{equation}

We now show by induction that the claim \eqref{diffquot11} holds.
Notice first that $A_\ell\to0$ as $\ell\to-\infty$ by \eqref{diffquot-1}, and that the series $\sum_{k=-\infty}^\infty A_k e^{-ka}$ is convergent: indeed, using once more \eqref{diffquot-1},
\begin{align*}
\sum_{k=-\infty}^\infty A_k e^{-ka}
&= \sum_{k=-\infty}^\infty \Delta_\tau(k\rho)e^{-ka}
\leq \frac{c}{\tau} \sum_{k=-\infty}^\infty e^{\lambda(k\rho)}e^{-ka}
\leq \frac{c}{\tau\rho} \int_{-\infty}^{\infty} e^{-\frac{a}{\rho}z}e^{\lambda(z)}\de z <\infty\,,
\end{align*}
the last integral being bounded by $c\rho$ as a consequence of \eqref{parameters4}--\eqref{parameters4b}.
Therefore there exists $\ell_0\leq0$, possibly depending on $\tau$ and $\rho$, such that for every $\ell\leq \ell_0$ 
\begin{equation*}
A_\ell \leq \frac{1}{16}, \qquad \sum_{k=-\infty}^{\ell} A_ke^{-ka} \leq \frac14\,.
\end{equation*}
Using these inequalities in \eqref{diffquot14}, we immediately get \eqref{diffquot11} for all $\ell\leq \ell_0$.
We now check the induction step: assuming that \eqref{diffquot11} holds for all $\ell\leq\bar{\ell}$, for some $\bar{\ell}\leq -1$, let us prove that \eqref{diffquot11} holds also for $\bar{\ell}+1$: we have
\begin{align*}
\sum_{k=-\infty}^{\bar{\ell}} A_k e^{-ka}
\xupref{diffquot11}{\leq} \frac{\cbar}{\rho} \sum_{k=-\infty}^{\bar{\ell}} e^{\lambda(k\rho)}e^{-2ka}
\leq \frac{c\cbar}{\rho^2} \int_{-\infty}^{0} e^{-\frac{2a}{\rho}z}e^{\lambda(z)}\de z
\leq \frac{c\cbar}{\rho}\,,
\end{align*}
for another uniform constant $c$ independent of $\rho$, $\tau$, and $\ell$ (the last inequality follows by \eqref{parameters4b}).
Then by \eqref{diffquot14}
\begin{align*}
A_{\bar{\ell}+1}
& \leq \frac{\cbar}{4} \Bigl(1+\frac{1}{\rho}e^{-a(\bar{\ell}+1)}\Bigr) e^{\lambda((\bar{\ell}+1)\rho)}
+ 4\cbar\rho A_{\bar{\ell}}e^{\lambda((\bar{\ell}+1)\rho)}
+ \cbar\rho e^{\lambda((\bar{\ell}+1)\rho)}\sum_{k=-\infty}^{\bar{\ell}}A_k e^{-ka} \\
& \leq \frac{\cbar}{4} \Bigl(1+\frac{1}{\rho}e^{-a(\bar{\ell}+1)}\Bigr) e^{\lambda((\bar{\ell}+1)\rho)}
+4\cbar^2e^{-a\bar{\ell}}e^{\lambda((\bar{\ell}+1)\rho)}
+ c\cbar^2 e^{\lambda((\bar{\ell}+1)\rho)}\,.
\end{align*}
By reducing $\bar{\rho}(L)$ if necessary, we therefore obtain \eqref{diffquot11}.

To conclude the proof, it only remains to prove the second claim \eqref{diffquot12}.
We again proceed by induction. Notice first that \eqref{diffquot12} holds for $\ell=0$, thanks to \eqref{diffquot11}.
Assume then that \eqref{diffquot12} holds for all $\ell\leq\bar{\ell}$, for some $\bar{\ell}\geq 0$, and let us prove that \eqref{diffquot12} holds also for $\bar{\ell}+1$: we have
\begin{align*}
\sum_{k=-\infty}^{\bar{\ell}} A_k e^{-ka}
& \leq \frac{\cbar}{\rho} \sum_{k=-\infty}^{-1} e^{\lambda(k\rho)}e^{-2ka}
+ \frac{\cbar}{\rho} \sum_{k=0}^{\bar{\ell}} e^{\lambda(k\rho)}e^{-ka} \\
& \leq \frac{c\cbar}{\rho^2} \int_{-\infty}^{0} e^{-\frac{2a}{\rho}z}e^{\lambda(z)}\de z
+  \frac{c\cbar}{\rho^2} \int_{0}^{(\bar{\ell}+1)\rho} e^{-\frac{a}{\rho}z}e^{\lambda(z)}\de z
\leq \frac{c\cbar}{\rho}\,,
\end{align*}
for another uniform constant $c$ independent of $\rho$, $\tau$ and $\ell$.
Then by \eqref{diffquot13}
\begin{align*}
A_{\bar{\ell}+1}
& \leq \frac{\cbar}{4} \Bigl(1+\frac{1}{\rho}e^{-a(\bar{\ell}+1)}\Bigr) e^{\lambda((\bar{\ell}+1)\rho)}
+ \cbar A_{\bar{\ell}} \int_{-\infty}^{(\bar{\ell}+1)\rho} e^{-\frac{a}{\rho}z}|h(z)|\de z
+ \cbar\rho e^{\lambda((\bar{\ell}+1)\rho)}\sum_{k=-\infty}^{\bar{\ell}}A_k e^{-ka} \\
& \leq \frac{\cbar}{4} \Bigl(1+\frac{1}{\rho}\Bigr) e^{\lambda((\bar{\ell}+1)\rho)}
+ c\cbar^2e^{\lambda(\bar{\ell}\rho)}
+ c\cbar^2 e^{\lambda((\bar{\ell}+1)\rho)}
\end{align*}
(where the integral is bounded by means of \eqref{eq:prelim1}). By eventually choosing $\bar{\rho}$ small enough, we conclude that also \eqref{diffquot11} holds.
\end{proof}

Having a bound on the Lipschitz constant of $h$ at hand, in order to obtain the desired decay \eqref{eq:asimp+} it is now convenient to write the equation \eqref{equation} in the following form:
\begin{equation} \label{finequation}
h(x) = \int_x^\infty h(z)\de z + (1+\rho)H(x)\,,
\end{equation}
with the remainder term given by
\begin{align} \label{finremainder}
H(x) &= \frac{1}{\rho} \int_{-\infty}^x \de y \int_{x+\rho\ln(1-e^{\frac{y-x}{\rho}})}^\infty  e^{-\frac{a}{\rho}z} K(e^{\frac{y-z}{\rho}},1) h(y)h(z) \de z \nonumber\\
& \qquad\qquad -\frac{\rho}{1+\rho}h(x) - \frac{1}{1+\rho}\int_x^\infty h(z)\de z\,.
\end{align}
We employ Lemma~\ref{lem:diffquot} to show that $H(x)$ decays faster than $e^{-x}$.

\begin{lemma}\label{lem:remainderfin}
There exist a constant $A(L,\rho)$ such that
\begin{equation*}
|H(x)| \leq A(L,\rho) e^{-\frac32(x-L\rho)} \qquad\text{for every }x\geq L\rho\,.
\end{equation*}
The constant $A(L,\rho)$ depends on $L$ and $\rho$ and can be made arbitrarily small by choosing $L$ sufficiently large and, in turn, $\rho$ small enough.
\end{lemma}

\begin{proof}
In order to prove the result we manipulate the expression \eqref{finremainder} of $H$ and we write it as the sum of five terms,
$H(x) = \sum_{i=1}^5 I_i$, which are explicitly given by
\begin{align*}
I_1 := \frac{1}{\rho}\int_{-\infty}^x \de y \int_{x}^\infty \expaz K(e^{\frac{y-z}{\rho}},1)h(y)h(z) \de z - \frac{1}{\rho}\int_{-\infty}^x \expay h(y)\de y \int_x^\infty h(z)\de z\,,
\end{align*}
\begin{align*}
I_2 := \frac{1}{\rho}\int_{x-\rho\ln2}^x \de z \int_{x+\rho\ln(1-e^{\frac{z-x}{\rho}})}^x \expaz K(e^{\frac{y-z}{\rho}},1)h(y)h(z)\de y\,,
\end{align*}
\begin{align*}
I_3 := \frac{1}{\rho}\int_{-\infty}^{x-\rho\ln2} \de z \int_{x+\rho\ln(1-e^{\frac{z-x}{\rho}})}^x \Bigl( \expaz K(e^{\frac{y-z}{\rho}},1) - e^{-\frac{a+1}{\rho}z}e^{\frac{y}{\rho}}\Bigr) h(y)h(z)\de y\,,
\end{align*}
\begin{align*}
I_4 := \frac{1}{\rho}\int_{-\infty}^{x-\rho\ln2} \de z \int_{x+\rho\ln(1-e^{\frac{z-x}{\rho}})}^x e^{-\frac{a+1}{\rho}z}e^{\frac{y}{\rho}} h(y)h(z)\de y
-h(x)\int_{-\infty}^{x-\rho\ln2}\expaz h(z)\de z\,,
\end{align*}
\begin{align*}
I_5 &:= \frac{1}{\rho}\int_{-\infty}^x \expay h(y)\de y \int_x^\infty h(z)\de z
+ h(x)\int_{-\infty}^{x-\rho\ln2}\expaz h(z)\de z \\
& \qquad\qquad -\frac{\rho}{1+\rho}h(x) - \frac{1}{1+\rho}\int_x^\infty h(z)\de z\,.
\end{align*}
Notice in particular that the first integral $I_1$ is over the region above the $y$-axis, the second term $I_2$ is over the domain $A_\rho$, and $I_3$, $I_4$ are integrals over $B_\rho$ (see Figure~\ref{fig:domains}). We proceed to estimate the five terms separately, for $x\geq L\rho$. For the first term we have
\begin{align} \label{fin1}
|I_1|
& \leq \frac{1}{\rho}\int_{-\infty}^x \int_{x}^\infty \Big| \expaz K(e^{\frac{y-z}{\rho}},1) -\expay \Big| |h(y)||h(z)| \de z \de y \nonumber\\
& \xupref{kernel3}{\leq} \frac{C}{\rho} \int_{-\infty}^x \int_{x}^\infty \expay e^{\frac{\delta}{\rho}(y-z)} |h(y)||h(z)| \de z \de y \nonumber\\
& \leq \frac{C}{\rho}\int_{-\infty}^x e^{\frac{\delta-a}{\rho}y}|h(y)|\de y \int_x^{\infty}e^{-\frac{\delta}{\rho}z}e^{-\frac12(z-L\rho)}\de z \nonumber\\
& \xupref{eq:prelim2}{\leq} C\rho e^{aL} e^{-\frac12(x-L\rho)}e^{-\frac{\delta}{\rho}x}\,.
\end{align}

We next consider the second integral, which is over the region $A_\rho$: by \eqref{kernelA} and Fubini's Theorem we obtain
\begin{align*}
|I_2|
& \leq \frac{C}{\rho}\int_{x-\rho\ln2}^x \de z \int_{x+\rho\ln(1-e^{\frac{z-x}{\rho}})}^x \expay |h(y)||h(z)|\de y \nonumber\\
& \leq \frac{C}{\rho}\int_{-\infty}^{x-\rho\ln2} \de y \,\expay|h(y)| \int_{x+\rho\ln(1-e^{\frac{y-x}{\rho}})}^x |h(z)|\de z \nonumber\\
& \qquad \qquad+ \frac{C}{\rho}\int_{x-\rho\ln2}^x\expay|h(y)|\de y \int_{x-\rho\ln2}^x|h(z)|\de z\,.
\end{align*}
To proceed we recall that for $\xi\in(x-\rho\ln2,x)$ we have $|h(\xi)|\leq Ce^{-\frac12(x-L\rho)}$; also bearing in mind the elementary inequality $|\ln(1-t)|\leq Ct$ for $t=e^{\frac{y-x}{\rho}}\in(0,\frac12)$ we get
\begin{align}\label{fin2}
|I_2|
& \leq Ce^{-\frac12(x-L\rho)} \int_{-\infty}^{x-\rho\ln2}\expay|h(y)| |\ln(1-e^{\frac{y-x}{\rho}})| \de y + Ce^{-(x-L\rho)}\int_{x-\rho\ln2}^x\expay\de y \nonumber\\
& \leq Ce^{-\frac12(x-L\rho)} \int_{-\infty}^{x}e^{\frac{y-x}{\rho}}\expay|h(y)| \de y + C\rho e^{-\frac{a}{\rho}x}e^{-(x-L\rho)}\nonumber\\
& \leq Ce^{-\frac12(x-L\rho)} \biggl( e^{-\frac{x}{2\rho}}\int_{-\infty}^{\frac{x}{2}}\expay|h(y)| \de y + \int_{\frac{x}{2}}^{x} \expay \de y \biggr) + C\rho e^{-\frac{a}{\rho}x}e^{-(x-L\rho)}\nonumber\\
& \xupref{eq:prelim1}{\leq} Ce^{-\frac12(x-L\rho)} \biggl( \rho e^{aL}e^{-\frac{x}{2\rho}} + \rho e^{-\frac{a}{2\rho}x} \biggr) + C\rho e^{-\frac{a}{\rho}x}e^{-(x-L\rho)}\,.
\end{align}

The third integral $I_3$ is instead over the region $B_\rho$: in this case using \eqref{regions3} one has
\begin{align}\label{fin3}
|I_3|
& \leq \frac{C}{\rho}\int_{-\infty}^{x-\rho\ln2} \de z \int_{x+\rho\ln(1-e^{\frac{z-x}{\rho}})}^x \expaz e^{\frac{1-\delta}{\rho}(y-z)} |h(y)||h(z)|\de y \nonumber\\
& \leq \frac{C}{\rho} e^{-\frac12(x-L\rho)} \int_{-\infty}^{x-\rho\ln2} \de z \, e^{\frac{\delta-a-1}{\rho}z}|h(z)| \int_{x+\rho\ln(1-e^{\frac{z-x}{\rho}})}^x e^{\frac{1-\delta}{\rho}y} \de y \nonumber\\
& \leq C e^{-\frac12(x-L\rho)}e^{-\frac{\delta}{\rho}x} \int_{-\infty}^{x-\rho\ln2} e^{\frac{\delta-a}{\rho}z}|h(z)| \de z
\xupref{eq:prelim2}{\leq} C\rho e^{aL}e^{-\frac12(x-L\rho)}e^{-\frac{\delta}{\rho}x}\,,
\end{align}
where in the third passage we used the inequality
\begin{equation*}
\int_{x+\rho\ln(1-e^{\frac{z-x}{\rho}})}^x e^{\frac{1-\delta}{\rho}y} \de y = \frac{\rho}{1-\delta}e^{\frac{1-\delta}{\rho}x} \Bigl( 1 - (1-e^{\frac{z-x}{\rho}})^{1-\delta} \Bigr)
\leq C\rho e^{\frac{1-\delta}{\rho}x}e^{\frac{z-x}{\rho}}\,,
\end{equation*}
which holds since $e^{\frac{z-x}{\rho}}\in(0,\frac12)$.

In order to bound the term $I_4$, we use the result in Lemma~\ref{lem:diffquot}, which gives in particular
\begin{equation*}
|h(y)-h(x)| \leq \frac{A_L}{\rho}|x-y|e^{-\frac12(x-L\rho)}
\qquad\text{for }y\in(x-\rho,x).
\end{equation*}
Therefore
\begin{align}\label{fin4}
|I_4|
& \leq \frac{1}{\rho}\int_{-\infty}^{x-\rho\ln2} \de z \,\expaz |h(z)| \int_{x+\rho\ln(1-e^{\frac{z-x}{\rho}})}^x e^{\frac{y-z}{\rho}} \big|h(y)-h(x)\big| \de y \nonumber\\
& \leq \frac{A_L e^{-\frac12(x-L\rho)}}{\rho^2} \int_{-\infty}^{x-\rho\ln2} \de z \,\expaz |h(z)| \int_{x+\rho\ln(1-e^{\frac{z-x}{\rho}})}^x e^{\frac{y-z}{\rho}}(x-y) \de y \nonumber\\
& \leq A_L e^{-\frac12(x-L\rho)} \int_{-\infty}^{x-\rho\ln2} \expaz e^{\frac{x-z}{\rho}} |h(z)| |\ln(1-e^{\frac{z-x}{\rho}})|^2 \de z \nonumber\\
& \leq C A_L e^{-\frac12(x-L\rho)} \int_{-\infty}^{x} \expaz e^{\frac{z-x}{\rho}} |h(z)| \de z \nonumber\\
& \leq C A_L e^{-\frac12(x-L\rho)} \biggl( \rho e^{aL}e^{-\frac{x}{2\rho}} + \rho e^{-\frac{a}{2\rho}x} \biggr)\,,
\end{align}
where the last inequality follows as in the last passages of \eqref{fin2}.

Finally, for the last term $I_5$ we have, recalling that $h$ satisfies the constraints \eqref{constraints},
\begin{align} \label{fin5}
|I_5|
& \leq \frac{1}{\rho}\int_{x}^\infty \expay |h(y)| \de y \int_x^\infty |h(z)|\de z + |h(x)|\int_{x-\rho\ln2}^{\infty}\expaz |h(z)| \de z \nonumber\\
& \xupref{parameters4}{\leq} \frac{1}{\rho}\int_{x}^\infty \expay e^{-\frac12(y-L\rho)}\de y \int_x^\infty e^{-\frac12(z-L\rho)}\de z + e^{-\frac12(x-L\rho)}\int_{x-\rho\ln2}^{\infty}\expaz e^{-\frac12(z-L\rho)} \de z \nonumber\\
& \leq C \bigl(1+\rho\bigr) e^{-\frac{a}{\rho}x} e^{-(x-L\rho)}
\leq Ce^{-aL}e^{-\frac{a}{\rho}(x-L\rho)}e^{-(x-L\rho)}\,.
\end{align}

By collecting \eqref{fin1}--\eqref{fin5}, we eventually obtain a bound on the function $H$ of the form
\begin{align*}
|H(x)|
\leq C_L\rho \Bigl( e^{-\frac{\delta}{\rho}x} + e^{-\frac{x}{2\rho}} + e^{-\frac{a}{2\rho}x} \Bigr)e^{-\frac12(x-L\rho)} + Ce^{-aL} e^{-\frac{a}{\rho}(x-L\rho)}e^{-(x-L\rho)}\,,
\end{align*}
for a uniform constant $C$, depending only on the kernel, and a constant $C_L$ possibly depending also on $L$. As we can assume that $\rho$ is so small that
\begin{equation*}
\frac{\delta}{\rho}>1\,, \qquad \frac{1}{2\rho}>1\,,\qquad\frac{a}{2\rho}>1\,,
\end{equation*}
the proof of the lemma is completed.
\end{proof}

We are now in position to prove the explicit decay of the solution $h$ at $\infty$.

\begin{theorem}[Decay at $\infty$] \label{thm:decay+}
	There exist $L_2\geq\bar{L}$ and $\rho_2:(L_2,\infty)\to(0,\rho_0)$ such that for every $L > L_2$ and $\rho\in(0,\rho_2(L))$ the solution $h$ to \eqref{equation}--\eqref{constraints} determined in Theorem~\ref{thm:fixedpoint} satisfies
	\begin{equation} \label{eq:decay+}
	\textstyle\frac14 e^{-(x-L\rho)} \leq h(x) \leq 2e^{-(x-L\rho)} \qquad\text{for all }x \geq L\rho
	\end{equation}
	and
	\begin{equation}\label{eq:decay+bis}
	h(x) = k_{L,\rho} e^{-x} + o(e^{-x}) \qquad\text{as }x\to\infty,
	\end{equation}
	for a positive constant $k_{L,\rho}$, depending on $L$ and $\rho$, with the property that $|k_{L,\rho}-1|$ can be made arbitrarily small by choosing $L$ large enough and, in turn, $\rho$ small enough (depending on $L$).
\end{theorem}

\begin{proof}
By integrating \eqref{finequation} we have that $h$ solves
\begin{equation} \label{eq:fin0}
h(x) = k_0 e^{-(x-L\rho)} + (1+\rho)H(x) -(1+\rho)\int_{L\rho}^x e^{y-x}H(y)\de y
\end{equation}
for $x\geq L\rho$. The constant $k_0$ can be computed explicitly:
\begin{equation*}
k_0 = \int_{L\rho}^\infty h(z)\de z = 1 - \int_{-\infty}^{L\rho} h(z)\de z\,,
\end{equation*}
and since by Theorem~\ref{thm:decay-}
\begin{align*}
\bigg| \int_{-\infty}^{L\rho} h(z)\de z \bigg| \leq 2\int_{-\infty}^{L\rho}e^{-\psi_\rho(z)}\de z
\xupref{lem:Q2}{\leq} 4e^{aL}\int_{-\infty}^{L\rho}e^{-\psi_\rho(z)} \expaz Q(z)\de z
\xupref{eq:derivativepsi}{=} 4\rho e^{aL}\,,
\end{align*}
we see that, for fixed $L$, $k_0\to1$ as $\rho\to0$.
	
We rewrite \eqref{eq:fin0} in the following form:
\begin{align} \label{eq:fin1}
h(x) = \biggl(k_0 - (1+\rho)e^{-L\rho}&\int_{L\rho}^\infty e^{y}H(y)\de y \biggr) e^{-(x-L\rho)} \nonumber\\
& \qquad + (1+\rho)H(x) + (1+\rho)\int_{x}^\infty e^{y-x}H(y)\de y\,.
\end{align}
The last two terms on the right hand side of \eqref{eq:fin1} decay faster than $e^{-x}$ as $x\to\infty$: indeed, by Lemma~\ref{lem:remainderfin} we have
\begin{align} \label{eq:fin2}
\bigg| (1+\rho)H(x) + (1+\rho)&\int_{x}^\infty e^{y-x}H(y)\de y \bigg| \nonumber\\
& \leq (1+\rho)A(L,\rho) \biggl( e^{-\frac32(x-L\rho)} + \int_x^\infty e^{y-x}e^{-\frac32(y-L\rho)}\de y \biggr) \nonumber\\
& \leq 3(1+\rho)A(L,\rho) e^{-\frac32(x-L\rho)}\,.
\end{align}
Moreover, using once more Lemma~\ref{lem:remainderfin} we obtain the bound
\begin{align} \label{eq:fin3}
\bigg|(1+\rho)e^{-L\rho}\int_{L\rho}^\infty e^{y}H(y)\de y \bigg|
\leq (1+\rho)A(L,\rho)\int_{L\rho}^\infty e^{-\frac12(y-L\rho)}\de y
= 2(1+\rho)A(L,\rho)\,.
\end{align}
The properties \eqref{eq:decay+} and \eqref{eq:decay+bis} follow now by combining \eqref{eq:fin1}--\eqref{eq:fin3}, recalling the property of $A(L,\rho)$ stated in Lemma~\ref{lem:remainderfin} and that $k_0\to1$ as $\rho\to0$.
\end{proof}

We conclude the paper with the proof of the main result, Theorem~\ref{thm:exist}.

\begin{proof}[Proof of Theorem~\ref{thm:exist}]
Let $L_2$ and $\rho_2(L)$, for $L>L_2$, be given by Theorem~\ref{thm:decay+}.
We can further select two monotone sequences $(L_n)_{n\geq3}$ and $(\rho_n)_{n\geq3}$ with the properties that
$$
L_n\to\infty\,, \qquad \rho_n\to0\,,\qquad \rho_n\in(0,\rho_2(L_n))\,,
$$
and
\begin{equation} \label{eq:fin4}
\lim_{n\to\infty}k_{L_n,\rho_n}\to1\,,
\end{equation} 
where $k_{L,\rho}$ is the constant appearing in Theorem~\ref{thm:decay+}.

We choose $\rho_*:=\rho_3$. If $\rho\in(0,\rho_*)$, then $\rho\in[\rho_{n+1},\rho_n)$ for some $n$.
For this value of $\rho$ we then select $h_\rho$ to be the solution to \eqref{equation}--\eqref{constraints} in $X_{L_n,\rho}$ given by Theorem~\ref{thm:fixedpoint}.
The continuity of $h_\rho$ follows from Remark~\ref{rm:continuity}, and moreover $h_\rho$ enjoys the decay estimates \eqref{eq:decay-} and \eqref{eq:decay+}--\eqref{eq:decay+bis}, proved in Theorem~\ref{thm:decay-} and Theorem~\ref{thm:decay+} respectively.
In particular, by \eqref{eq:decay+bis} we can write
\begin{equation*}
h_\rho(x) = k_\rho e^{-x} +o(e^{-x}) \qquad\text{as }x\to\infty,
\end{equation*}
with $k_\rho\to1$ as $\rho\to0$ thanks to \eqref{eq:fin4}.
By the change of variables \eqref{variables}, we obtain a family of solutions to \eqref{eq:selfsim2} with the desired properties.
\end{proof}


\bigskip
\bigskip
\noindent
{\bf Acknowledgments.}
The authors acknowledge support through the CRC 1060 \textit{The mathematics of emergent effects} at the University of Bonn that is funded through the German Science Foundation (DFG).

\bibliographystyle{siam}
\bibliography{bibliography}

\end{document}